%% file: draft.tex
\documentclass{article}
\usepackage{epsfig}
\usepackage{amssymb,amsmath,amsthm,bbm,color,array,multicol}
\usepackage{npbayes}

\begin{document}

%\inserttype[ba0001]{article}
\author{%Broderick {\it et al.}}{%
Tamara Broderick, %\footnote{Statistics Department, University of California, Berkeley, CA},
Jim Pitman, %\footnote{Statistics Department, University of California, Berkeley, CA},
Michael I. Jordan %\footnote{Statistics Department, University of California, Berkeley, CA}
}
\title{Feature allocations, probability functions, and paintboxes}
\maketitle

\begin{abstract}
The problem of inferring a clustering of a data set has been the subject of
much research in Bayesian analysis, and there currently exists a solid mathematical 
foundation for Bayesian approaches to clustering.  In particular, the class of 
probability distributions over partitions of a data set has been characterized 
in a number of ways, including via exchangeable partition probability functions 
(EPPFs) and the Kingman paintbox.  Here, we develop a generalization of 
the clustering problem, called feature allocation, where we allow each data point 
to belong to an arbitrary, non-negative integer number of groups, now called 
features or topics.  We define and study an ``exchangeable feature probability 
function'' (EFPF)---analogous to the EPPF in the clustering setting---for certain 
types of feature models.  Moreover, we introduce a ``feature paintbox'' 
characterization---analogous to the Kingman paintbox for clustering---of 
the class of exchangeable feature models.  We provide a further characterization
of the subclass of feature allocations that have EFPF representations.
%We highlight an analogy between clusterings, considered as a subclass of feature allocations, and
%feature allocations with EFPF representations.

%\keywords{feature, feature allocation, paintbox, feature frequency model, Indian buffet process}
\end{abstract}

%%%%%%%%%%%%%%%%%%%%%%%%%%%%
\section{Introduction} \label{sec:introduction}
%%%%%%%%%%%%%%%%%%%%%%%%%%%%

Exchangeability has played a key role in the development of 
Bayesian analysis in general and Bayesian nonparametric analysis 
in particular.  Exchangeability can be viewed as asserting that 
the indices used to label the data points are irrelevant for 
inference, and as such is often a natural modeling assumption.  
Under such an assumption, one is licensed by de Finetti's 
theorem~\citep{de_finetti:1931:funzione} to propose the existence of an 
underlying parameter that renders the data conditionally 
independent and identically distributed (iid) and to place 
a prior distribution on that parameter.  Moreover, the 
theory of infinitely exchangeable sequences has advantages 
of simplicity over the theory of finite exchangeability, 
encouraging modelers to take a nonparametric stance in which the 
underlying ``parameter'' is infinite dimensional.  Finally, the 
development of algorithms for posterior inference is often greatly 
simplified by the assumption of exchangeability, most notably in 
the case of Bayesian nonparametrics, where models based on the 
Dirichlet process and other combinatorial priors became useful 
tools in practice only when it was realized how to exploit exchangeability 
to develop inference procedures~\citep{escobar:1994:estimating}.

The connection of exchangeability to Bayesian nonparametric modeling
is well established in the case of models for clustering.  The
goal of a clustering procedure is to infer a partition of the
data points.  In the Bayesian setting, one works with random 
partitions, and, under an exchangeability assumption, the distribution 
on partitions should be invariant to a relabeling of the data 
points.  The notion of an exchangeable random partition has been 
formalized by Kingman, Aldous, and
others~\citep{kingman:1978:representation,aldous:1985:exchangeability},
and has led to the definition of an \emph{exchangeable partition 
probability function} (EPPF)~\citep{pitman:1995:exchangeable}.
The EPPF is a mathematical function of the cardinalities of the 
groups in a partition.  Exchangeability of the random partition is 
captured by the requirement that the EPPF be a symmetric function 
of these cardinalities.  Furthermore, the exchangeability of a partition 
can be related to the exchangeability of a sequence of random variables 
representing the assignments of data points to clusters, for which a 
de Finetti mixing measure necessarily exists.  This de Finetti measure is
known as the \emph{Kingman paintbox}~\citep{kingman:1978:representation}.
The relationships among this circle of ideas are well understood:
it is known that there is an equivalence among the class of exchangeable 
random partitions, the class of random partitions that possess an 
EPPF, and the class of random partitions generated by a Kingman 
paintbox; see \citet{pitman:2006:combinatorial} for an overview
of these relations.
A specific example of these relationships is given by 
the Chinese restaurant process and the Dirichlet process, but several 
other examples are known and have proven useful in Bayesian nonparametrics.

Our focus in the current paper is on an alternative to clustering models 
that we refer to as \emph{feature allocation models}.  While in a clustering 
model each data point is assigned to one and only one class, in a feature 
allocation model each data point can belong to multiple groups.  It is often 
natural to view the groups as corresponding to traits or features, such that 
the notion that a data point belongs to multiple groups corresponds to the point 
exhibiting multiple traits or features.  A Bayesian feature allocation model 
treats the feature assignments for a given data point as random and subject 
to posterior inference.  A nonparametric Bayesian feature allocation model takes
the number of features to also be random and subject to inference.

Research on nonparametric Bayesian feature allocation has been based around 
a single prior distribution, the Indian buffet process of~\citet{griffiths:2006:infinite},
which is known to have the beta process as its underlying de Finetti
measure~\citep{thibaux:2007:hierarchical}.  There does not yet exist
a general definition of exchangeability for feature allocation models,
nor counterparts of the EPPF or the Kingman paintbox.

In this paper we supply these missing constructions.  We provide a
rigorous treatment of exchangeable feature allocations (in \mysec{feature}
and \mysec{label}).  In \mysec{efpf} we define a notion of \emph{exchangeable 
feature probability function} (EFPF) that is the analogue for feature 
allocations of the EPPF for clustering.  We then proceed to define
a \emph{feature paintbox} in \mysec{paintbox}.  Finally, in \mysec{freq} 
we discuss a class of models that we refer to as \emph{feature frequency models}
for which the construction of the feature paintbox is particularly
straightforward, and we discuss the important role that feature frequency models 
play in the general theory of feature allocations.

The Venn diagram shown in \fig{venn_summary} is a useful guide for
understanding our results, and the reader may wish to consult this
diagram in working through the paper.
%%%
\begin{figure}
	\centerline{
		\scalebox{0.45}{
			\input{venn_summary.pstex_t}
		}
	}
\caption{\label{fig:venn_summary} A summary of the relations described in
this paper. Rounded rectangles represent classes with the following abbreviations:
RP for random partition, FA for random feature allocation, EPPF for
exchangeable partition probability function, EFPF for exchangeable feature probability
function. The large black dots represent particular models with the
following abbreviations: CRP for Chinese restaurant process, IBP for
Indian buffet process. The two-feature example refers to \ex{two_feat} with the
choice $p_{11} p_{00} \ne p_{10} p_{01}$.
}
\end{figure}
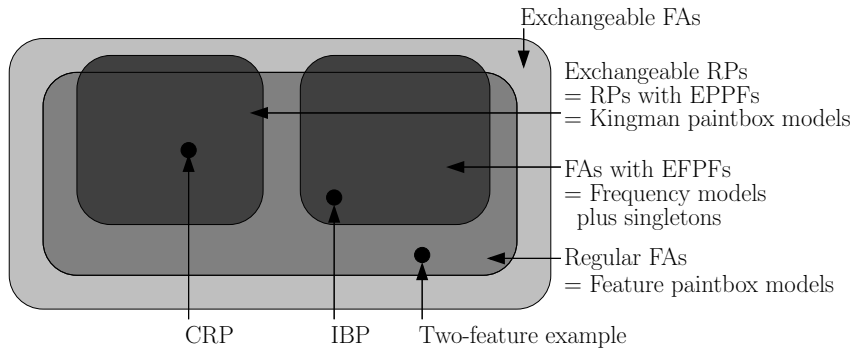
%%%
As shown in the diagram, random partitions (RPs) are a special case of 
random feature allocations (FAs), and previous work on random partitions
can be placed within our framework.  Thus, in the diagram, we have depicted
the equivalence already noted of exchangeable RPs, RPs that possess an EPPF, 
and Kingman paintboxes.  We also see that random feature allocations have 
a somewhat richer structure: the class of FAs with EFPFs is not the same 
as those having an underlying feature paintbox.  But the class of EFPFs 
is characterized in a different way; we will see that
the class of feature allocations with EFPFs
is equivalent to the class of 
FAs obtained from feature frequency models together with singletons of a certain distribution. Indeed, we will find that
the class of clusterings with EPPFs is, in this way, analogous to the class of feature allocations with
EFPFs when both are considered as subclasses of the general class of feature allocations.
The diagram also shows several examples 
that we use to illustrate and develop our theory.

%%%%%%%%%%%%%%%%%%%%%%%%%%%%
\section{Feature allocations} \label{sec:feature}
%%%%%%%%%%%%%%%%%%%%%%%%%%%%

We consider data sets with $N$ points and let the points be indexed
by the integers $[N] := \{1,2,\ldots,N\}$.  We also explicitly allow
$N = \infty$, in which case the index set is $\mathbb{N} = \{1,2,3,\ldots\}$.
For our discussion of feature allocations and partitioning it is sufficient
to focus on the indices rather than the data points; thus, we will be
discussing models for collections of subsets of $[N]$ and $\mathbb{N}$.

Our introduction to feature allocations follows \citet{broderick:2012:clusters}.
We define a \emph{feature allocation} $\detfa_{N}$ of $[N]$ to be a multiset
of non-empty subsets of $[N]$ called \emph{features},
such that no index $n$ belongs to infinitely many features.
We write $\detfa_{N} = \{\blockfa_{1},\ldots,\blockfa_{K}\}$, where $K$ is
the number of features.  An example feature allocation of 
$[6]$ is $\detfa_{6} = \{\{2,3\},\{2,4,6\},\{3\},\{3\},\{3\}\}$. 
Similarly, a feature allocation $\detfa_{\infty}$ of $\mathbb{N}$
is a multiset of non-empty subsets of $\mathbb{N}$ such that no 
index $n$ belongs to infinitely many features.
The total number of features in this case may be infinite, in which case
we write
$\detfa_{\infty} = \{\blockfa_{1},\blockfa_{2},\ldots\}$. An example 
feature allocation of $\mathbb{N}$ is 
$\detfa_{\infty} = \{\{n: n \textrm{ is prime}\}, \{n: n \textrm{ is not divisible by two}\}\}$.
Finally, we may have $K=0$, and $\detfa_{\infty} = \emptyset$ is a valid feature allocation.

A \emph{partition} is a special case of a feature allocation for which
the features are restricted to be mutually exclusive and exhaustive.
The features of a partition are often referred to as \emph{blocks} or \emph{clusters}.
We note that a partition is always a feature allocation, but the converse statement
does not hold in general; neither of the examples given
above ($\detfa_{6}$ and $\detfa_{\infty}$) are partitions.

We now turn to the problem of defining exchangeable feature allocations, extending 
previous work on exchangeable random partitions~\citep{aldous:1985:exchangeability}.
Let $\randfasp_{N}$ be the space of all feature allocations of $[N]$.
A \emph{random feature allocation} $\randfa_{N}$ of $[N]$ is a random 
element of $\randfasp_{N}$.  Let $\perm: \mathbb{N} \rightarrow \mathbb{N}$ be a
finite permutation. That is, for some finite value $N_{\perm}$, we have 
$\perm(n) = n$ for all $n > N_{\sigma}$. Further, for any feature 
$\blockfa \subset \mathbb{N}$, denote the permutation applied to the feature as follows:
$\perm(\blockfa) := \{\perm(n): n \in \blockfa\}$. For any feature allocation $F_{N}$,
denote the permutation applied to the feature allocation as follows:
$\perm(\randfa_{N}) := \{\perm(\blockfa): \blockfa \in \randfa_{N}\}$.
Finally, let $\randfa_{N}$ be a random feature allocation of $[N]$. 
Then we say that a random feature allocation $\randfa_{N}$ is \emph{exchangeable} 
if $\randfa_{N} \eqd \sigma(\randfa_{N})$ for every permutation of $[N]$.

In addition to exchangeability, we also require our distributions on feature 
allocations to exhibit a notion of coherence across different ranges 
of the index.  Intuitively, we often imagine the indices as denoting
time, and it is natural to suppose that the randomness at time $n$ 
is coherent with the randomness at time $n+1$. More formally, we say 
that a feature allocation $\detfa_{M}$ of $[M]$ is a \emph{restriction} 
of a feature allocation $\detfa_{N}$ of $[N]$ for $M < N$ if
$$
	\detfa_{M} = \{\blockfa \cap [M]: \blockfa \in \detfa_{N}, \blockfa \cap [M] \neq \emptyset\}.
$$
Let $\restfa_{N}(\detfa_{M})$ be the set of all feature allocations of $[N]$
whose restriction to $[M]$ is $\detfa_{M}$.

Let $\mbp$ denote a probability measure on some probability space supporting
$(\randfa_{n})$.
We say that the sequence
of random feature allocations $(\randfa_{n})$ is \emph{consistent in distribution}
if
for all $M$ and
$N$ such that $M < N$, we have
$$
	%\label{eq:consistency_in_dist}
	\mbp(\randfa_{M} = \detfa_{M}) = \sum_{\detfa_{N} \in \restfa_{N}(\detfa_{M})} \mbp(\randfa_{N} = \detfa_{N}).
$$
We say that the sequence $(\randfa_{n})$ is \emph{strongly consistent} if
for all $M$ and
$N$ such that $M < N$, we have
$$
	\randfa_{N} \stackrel{a.s.}{\in} \restfa_{N}(\randfa_{M}).
$$
Given any $(\randfa_{n})$ that is consistent in distribution, the 
Kolmogorov extension theorem implies that we can construct a sequence of random
feature allocations that is strongly consistent and has the same finite
dimensional distributions. So henceforth we simply use the term 
``consistency'' to refer to strong consistency.

With this consistency condition, we can define a random feature allocation
$\randfa_{\infty}$
of $\mathbb{N}$ as a consistent sequence of finite feature allocations.
Thus $\randfa_{\infty}$ may be thought of as a random
element of the space of such sequences: $\randfa_{\infty} = (\randfa_{n})_{n=1}^{\infty}$.
We say that $\randfa_{N}$ is a restriction of $\randfa_{\infty}$ to $[N]$ when it
is the $N$th element in this sequence.
We let $\randfasp_{\infty}$ denote the space of consistent feature allocation sequences, of which
each random feature allocation is a random element. The sigma field
associated with this space is generated by the finite-dimensional sigma fields
of the restricted random feature allocations $\randfa_{n}$.

We say that $\randfa_{\infty}$ is exchangeable if
$\randfa_{\infty} \eqd \sigma(\randfa_{\infty})$ for every finite permutation
$\perm$. That is, for every permutation $\perm$ that changes no indices above $N$ for some
$N < \infty$, we
require $\randfa_{N} \eqd \sigma(\randfa_{N})$, where $\randfa_{N}$ is the
restriction of $\randfa_{\infty}$ to $[N]$.

%%%%%%%%%%%%%%%%%%%%%%%%%%%%
\section{Labeling features} \label{sec:label}
%%%%%%%%%%%%%%%%%%%%%%%%%%%%

Now that we have defined consistent, exchangeable random
feature allocations, we want to characterize the class
of all distributions on these allocations. We begin by considering some alternative
representations of the feature allocation that are not merely useful, but indeed
key to some of our later results.

A number of authors have made use of matrices as a way of representing 
feature allocations~\citep{griffiths:2006:infinite,thibaux:2007:hierarchical,doshi:2009:variational}. This representation, while
a boon for intuition in some regards, requires care because a matrix
presupposes an order
on the features, which is not a part of the feature allocation a priori. 
We cover this distinction in some
detail next.

We start by defining an \emph{a priori labeled feature allocation}.
Let $\hat{F}_{N,1}$ be the
collection of indices in $[N]$ with feature 1, let $\hat{F}_{N,2}$
be the collection of indices in $[N]$ with feature 2, etc.
Here, we think of a priori labels as being the ordered, positive natural numbers.
This specification
is different from (a priori unlabeled) feature allocations as defined above since there
is nothing to distinguish the features in a feature allocation other than, potentially,
the members of a feature.
Consider the following analogy: an a priori labeled feature allocation is to 
a feature allocation as a classification is to a clustering. Indeed, when each index
$n$ belongs to exactly one feature in an a priori feature allocation, feature 1 is just
class 1, feature 2 is class 2, and so on.

Another way to think of an a priori labeled feature allocation of $[N]$ is as a matrix of
$N$ rows filled with zeros and ones. Each column is associated with a feature.
The $(n,k)$ entry in the matrix is one if index $n$ is in feature $k$ and zero otherwise.
However, just as---contrary to the classification case---we do not know the ordering of clusters
in a clustering a priori, we do not a priori know the ordering of features in a feature allocation.
To make use of a matrix representation for a feature allocation, we will need to introduce or find such an order.

The reasoning above suggests that introducing an order for features in a feature allocations would be
useful. The next example illustrates 
that the probability $\mbp(\randfa_{N} = \detfa_{N})$ in some sense
undercounts features when they contain exactly 
the same indices: e.g., $\blockfa_{j} = \blockfa_{k}$ for some $j \ne k$. 
This fact will suggest to us that it is not merely useful, but indeed a key point 
of our theoretical development, to introduce an ordering on features.

%~~~~~~~~~~~~~~~~~~~~~~~~~~~~~~~~~~~~~~~~~~
\begin{example}[A Bernoulli, two-feature allocation] \label{ex:two_bern}

Given $q_{A}, q_{B} \in (0,1)$,
draw $\indpart_{n,A} \iid \bern(q_{A})$ and $\indpart_{n,B} \iid \bern(q_{B})$, independently,
and construct the random feature allocation by collecting those
indices with successful draws:
$$
	\randfa_{N} := \{\{n: n \le N, \indpart_{n,A} = 1\}, \{n: n \le N,
\indpart_{n,B} = 1\}\}.
$$
One caveat here is that if either of the two sets in the multiset $\randfa_{N}$
is empty, we do not include it in the allocation. Note that calling the
features $A$ and $B$ was merely for the purposes of construction, and in defining
$\randfa_{N}$, we have lost all feature labels. So $\randfa_{N}$ is a feature allocation,
not an a priori labeled feature allocation.

Then the probability of the feature allocation $\randfa_{5} = \detfa_{5} := \{\{2,3\},\{2,3\}\}$ is
$$
	q_{A}^{2} (1-q_{A})^{3} q_{B}^{2} (1-q_{B})^{3},
$$
but the probability of the feature allocation $\randfa_{5} = \detfa'_{5} := \{\{2,3\},\{2,5\}\}$ is
$$
	2 q_{A}^{2} (1-q_{A})^{3} q_{B}^{2} (1-q_{B})^{3}.
$$
The difference is that in the latter case the features can be distinguished, and so
we must account for the two possible pairings of features to frequencies $\{q_{A},q_{B}\}$.

Now, instead, let $\labrandfa_{N}$ be $\randfa_{N}$ with the features ordered 
uniformly at random amongst all possible feature orderings.
There is just a single possible ordering of $\detfa_{5}$, so
the probability of $\labrandfa_{5} = \labdetfa_{5} := (\{2,3\}, \{2,3\})$ is again
$$
	q_{A}^{2} (1-q_{A})^{3} q_{B}^{2} (1-q_{B})^{3}.
$$
However, there are two orderings of $\detfa'_{5}$, each of which
is equally likely. The
probability of $\labrandfa_{N} = \labdetfa'_{5} := (\{2,5\}, \{2,3\})$ is
$$
	q_{A}^{2} (1-q_{A})^{3} q_{B}^{2} (1-q_{B})^{3}.
$$
The same holds for the other ordering.
\end{example}
%~~~~~~~~~~~~~~~~~~~~~~~~~~~~~~~~~~~~~~~~~~

This example suggests that there are combinatorial factors that must
be taken into account when working with the distribution of $\randfa_{N}$
directly.  The example also suggests that we can avoid the need to specify such factors by instead 
working with a suitable randomized ordering of the random feature allocation 
$\randfa_{N}$.  We achieve this ordering in two steps.

The first step involves ordering the features via a procedure that 
we refer to as \emph{order-of-appearance labeling}. The basic idea
is that we consider data indices $n = 1, 2, 3$, and so on in order.
Each time a new data point arrives, we examine the features associated
with that data point. Each time we see a new feature, we label it with 
the lowest available feature label from $k = 1, 2, \ldots$.

In practice, the order-of-appearance scheme requires some auxiliary
randomness since each index $n$ may belong to zero, one, or many different 
features (though the number must be finite). When multiple features
first appear for index $n$, we order them uniformly at random.
That simple idea is explained in full detail as follows.
Recursively suppose that there are $K$ features among the
indices $[N-1]$. Trivially there are zero features when no indices
have been seen yet.  Moreover, we suppose
that we have features with labels $1$ through $K$
if $K \ge 1$, and if $K = 0$, we have no features.
If features remain without labels,
there exists some minimum
index $n$ in the data indices such that $n \notin \bigcup_{k=1}^{K} \blockfa_{k}$, where the
union is $\emptyset$ if $K = 0$.
It is possible that no features contain $n$.
So we further note that there exists some minimum index $m$ such
that $m \notin \bigcup_{j=1}^{K} \blockfa_{j}$ but $m$ is contained in some feature
of the allocation. 
By construction, we must have $m \ge N$.
Let $K_{m}$ be the number of features
containing $m$; $K_{m}$ is finite by definition of a feature allocation.
Let $(U_{k})$ denote a sequence of iid uniform random variables,
independent of the random feature allocation.
Assign $U_{K+1}, \ldots, U_{K + K_{m}}$ to these new features and
determine their order of appearance by the order of these random variables.
While features remain to be labeled,
continue the recursion with $N$ now equal to $m$ and $K$ now equal to $K + K_{m}$.

%~~~~~~~~~~~~~~~~~~~~~~~~~~~~~~~~~~~~~~~~~~
\begin{example}[Feature labeling schemes]
\label{ex:oa_6}
Consider the feature allocation
\begin{equation}
	\label{eq:feat_6_oa}
	\detfa_{6} = \{\{2,5,4\},\{3,4\},\{6,4\},\{3\},\{3\}\}.
\end{equation}
And consider the random variables
$$
	U_{1}, U_{2}, U_{3}, U_{4}, U_{5} \iid \unif[0,1].
$$
We see from $\detfa_{6}$ that index $1$ has no features. 
Index $2$ has exactly one feature, so we assign this feature, $\{2,5,4\}$,
to have order-of-appearance label 1. While $U_{1}$ is associated with
this feature, we do not need to break any ties at this point, so it has no effect.

Index $3$ is associated with three features. We associate each feature
with exactly one of $U_{2}, U_{3},$ and $U_{4}$ (the next three available $U_{k}$).
For instance, pair $\{3,4\}$ with $U_{2}$, $\{3\}$ with $U_{3}$, and the other
$\{3\}$ with $U_{4}$. Suppose it happens that $U_{3} < U_{2} < U_{4}$.
Then the feature $\{3\}$ paired with $U_{3}$ receives label 2 (the next available
order-of-appearance label).
The feature $\{3,4\}$ receives label 3. And the feature $\{3\}$
paired with $U_{4}$ receives label 4.

Index $4$ has three features, but $\{2,5,4\}$ and $\{3,4\}$ are already labeled.
So the only remaining feature, $\{6,4\}$, receives the next available 
order-of-appearance label: 5. $U_{5}$ is associated with this feature, but
since we do not need to break ties here, it has no effect.
Indices $5$ and $6$ belong to already-labeled features.

So the features can be listed with order-of-appearance indices as
\begin{equation}
	\label{eq:ex_A_oa}
	A_{1} = \{2,5,4\}, A_{2} = \{3\}, A_{3} = \{3,4\}, A_{4} = \{3\}, A_{5} = \{6,4\}.
\end{equation} 
Let $\indfao_{n}$ indicate the set of order-of-appearance feature labels
for the features to which index $n$ belongs;
i.e., if the features are labeled according to order of appearance
as in \eq{ex_A_oa}, then $\indfao_{n} = \{k: n \in A_{k}\}$.
By definition of a feature allocation, $\indfao_{n}$ must
have finite cardinality.
The order-of-appearance labeling gives
$
\indfao_{1} = \emptyset,
\indfao_{2} = \{1\},
\indfao_{3} = \{2,3,4\},
\indfao_{4} = \{1,3,5\},
\indfao_{5} = \{1\},
\indfao_{6} = \{5\}.
$

%%%
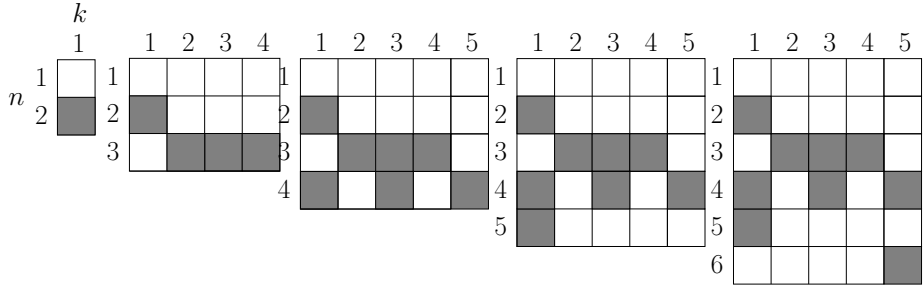
\begin{figure}
	%\centerline{
		\begin{minipage}[t]{0.10\textwidth}
			\vspace{0pt}
			\scalebox{0.5}{
				\input{mx_feat2.pstex_t}
			}
		\end{minipage}
		\begin{minipage}[t]{0.18\textwidth}
			\vspace{11pt}
			\scalebox{0.5}{
				\input{mx_feat3.pstex_t}
			}
		\end{minipage}
		\begin{minipage}[t]{0.23\textwidth}
			\vspace{11pt}
			\scalebox{0.5}{
				\input{mx_feat4.pstex_t}
			}
		\end{minipage}
		\begin{minipage}[t]{0.23\textwidth}
			\vspace{11pt}
			\scalebox{0.5}{
				\input{mx_feat5.pstex_t}
			}
		\end{minipage}
		\begin{minipage}[t]{0.23\textwidth}
			\vspace{11pt}
			\scalebox{0.5}{
				\input{mx_feat6.pstex_t}
			}
		\end{minipage}
	%}
\caption{\label{fig:mx_feat} Order-of-appearance binary matrix representations of the
sequence of feature allocations on $[2], [3], [4], [5],$ and $[6]$ found by restricting $\detfa_{6}$ in \ex{oa_6}. Rows correspond to
indices $n$, and columns correspond to order-of-appearance feature
labels $k$. A gray square indicates a 1 entry, and a white square indicates
a 0 entry. $\indfao_{n}$, the set of order-of-appearance
feature assignments of index $n$,
is easily read off from the matrix as the set of columns with entry
in row $n$ equal to 1.
}
\end{figure}
%%%

Order-of-appearance labeling is well-suited for matrix
representations of feature allocations. The rows of the matrix correspond
to indices $n$ and the columns correspond to features with order-of-appearance
labels $k$.
The matrix representation of the order-of-appearance labeling and 
resulting feature assignments $(\indfao_{n})$ for $n \in [6]$ is
depicted in \fig{mx_feat}.
\end{example}
%~~~~~~~~~~~~~~~~~~~~~~~~~~~~~~~~~~~~~~~~~~

Note that when the feature
allocation is a partition, there is exactly one feature containing any $m$,
so this scheme reduces to the order-of-appearance scheme for
cluster labeling.

Consider an exchangeable feature allocation $\randfa_{\infty}$.
Give order-of-appearance labels to the features of this allocation,
and let $\indfao_{n}$ be the set of feature labels for features containing $n$.
So $\indfao_{n}$ is a random finite subset of $\mathbb{N}$. It can be
thought of as a simple point process on $\mathbb{N}$; a discussion
of measurability of such processes may be found in
\citet[][p.\ 178]{kallenberg:2002:foundations}. Our process is even simpler
than a simple point process as it is globally finite rather than merely locally
finite.

Note that $(\indfao_{n})_{n=1}^{\infty}$ is not
necessarily exchangeable. For instance,
consider again \ex{two_bern}.
If $\indfao_{1}$ is non-empty, $1 \in \indfao_{1}$ with probability one.
If $\indfao_{2}$ is non-empty,
with positive probability it may not contain $1$.
To restore exchangeability we extend an idea due to \citet{aldous:1985:exchangeability} 
in the setting of random partitions, associating to each feature a draw 
from a uniform random variable on $[0,1]$.  Drawing these random variables
independently we maintain consistency across different values of $N$.
We refer to these random variables as \emph{uniform random feature labels}.

Note that the use of a uniform distribution is for convenience; we simply 
require that features receive distinct labels with probability one, so 
any other continuous distribution would suffice.  We also note that in 
a full-fledged model based on random feature allocations these labels
often play the role of parameters and are used in defining the likelihood.
For further discussion of such constructions, see \citet{broderick:2012:clusters}.

Thus, let $(\randlabfa_{k})$ be a sequence of iid uniform random variables,
independent of both $(U_{k})$ and $\randfa_{\infty}$. Construct a new feature
labeling by taking the feature labeled $k$ in the order-of-appearance labeling
and now label it $\randlabfa_{k}$. In this case, let $\indfar_{n}$ denote the set of feature
labels for features to which $n$ belongs. Call this a \emph{uniform random
labeling}. $\indfar_{n}$ can be
thought of as a (globally finite) simple point process on $[0,1]$.
Again, we refer the reader to \citet[][p.\ 178]{kallenberg:2002:foundations}
for a discussion of measurability.

% order of appearance still useful before random because
% gives method for us to assign random labels

%~~~~~~~~~~~~~~~~~~~~~~~~~~~~~~~~~~~~~~~~~~
\begin{example}[Feature labeling schemes (continued)] \label{ex:feat_6_unif_rand}
Again consider the feature allocation
$$
	\detfa_{6} = \{\{2,5,4\},\{3,4\},\{6,4\},\{3\},\{3\}\}.
$$
Now consider the random variables
$$
	U_{1}, U_{2}, U_{3}, U_{4}, U_{5}, \randlabfa_{1}, \randlabfa_{2}, \randlabfa_{3}, \randlabfa_{4}, \randlabfa_{5}
 \iid \unif[0,1].
$$
Recall from \ex{oa_6} that $U_{1}, \ldots, U_{5}$ gave us the order-of-appearance
labeling of the features. This labeling allowed us to index the features
as in \eq{ex_A_oa}, copied here:
\begin{equation}
	A_{1} = \{2,5,4\}, A_{2} = \{3\}, A_{3} = \{3,4\}, A_{4} = \{3\}, A_{5} = \{6,4\}.
\end{equation}

%%%
\begin{figure}
	\centerline{
		\scalebox{0.5}{
			\input{rand_lab.pstex_t}
		}
	}
\caption{\label{fig:rand_lab} An illustration of the uniform random feature labeling in \ex{feat_6_unif_rand}. The top rectangle is the unit interval. The uniform random labels
are depicted along the interval
with vertical dotted lines at their locations. The indices $[6]$ are shown to the
left. A black circle shows appears when an index occurs in the feature with a given label.
The matrix representations of this feature allocation in \fig{rmx_feat} can be recovered
from this plot.
}
\end{figure}
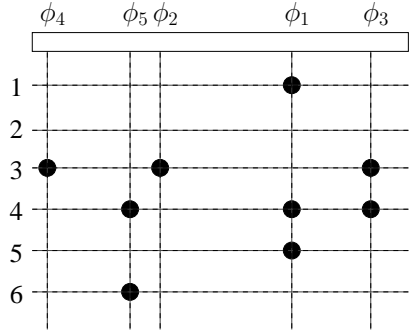
%%%

With this order-of-appearance labeling in hand, we can assign a uniform random 
label to each feature. In particular, we assign the uniform random label $\randlabfa_{k}$
to the feature with order-of-appearance label $k$: $A_{1} = \{2,5,4\}$ gets label $\randlabfa_{1}$,
$A_{2} = \{3\}$ gets label $\randlabfa_{2}$, $A_{3} = \{3,4\}$ gets label $\randlabfa_{3}$,
$A_{4} = \{3\}$ gets label $\randlabfa_{4}$, and $A_{5} = \{6,4\}$ gets label $\randlabfa_{5}$.
Let $\indfar_{n}$ indicate the set of uniform random feature labels
for the features to which index $n$ belongs.
The uniform random labeling gives
\begin{equation}
	\label{eq:feat_6_unif_rand_lab}
	\indfar_{1} = \emptyset,
	\indfar_{2} = \{\randlabfa_1\},
	\indfar_{3} = \{\randlabfa_2,\randlabfa_3,\randlabfa_4\},
	\indfar_{4} = \{\randlabfa_1,\randlabfa_3,\randlabfa_5\},
	\indfar_{5} = \{\randlabfa_1\},
	\indfar_{6} = \{\randlabfa_5\}.
\end{equation}
\end{example}
%~~~~~~~~~~~~~~~~~~~~~~~~~~~~~~~~~~~~~~~~~~

%%%
\begin{lemma}
\label{lem:unif_rand_exch}
Give the features of an exchangeable feature allocation $\randfa_{\infty}$ uniform random labels,
and let $\indfar_{n}$ be the set of feature labels for features containing $n$.
So $\indfar_{n}$ is a random finite subset of $[0,1]$.
Then the sequence
$(\indfar_{n})_{n=1}^{\infty}$ is exchangeable.
\end{lemma}
%%%

\begin{proof}
Note that $(\indfar_{n})_{n=1}^{\infty} = g((\randlabfa_{k})_{k}, (U_{k})_{k}, \randfa_{\infty})$ for some measurable function $g$. 
So, for any finite permutation $\sigma$, we have that 
$(\indfar_{\sigma(n)})_{n} = g((\randlabfa_{\tau(k)})_{k}, (U_{k})_{k}, \sigma(\randfa_{\infty}))$
where $\tau$ is a finite permutation that is a function of $\rho$, $(U_{k})$, $\sigma$, and $\randfa_{\infty}$.
Now
$$
	((\randlabfa_{\tau(k)})_{k}, (U_{k})_{k}, \sigma(\randfa_{\infty}))
		\eqd ((\randlabfa_{k})_{k}, (U_{k})_{k}, \sigma(\randfa_{\infty}))
$$
since the iid sequence $(\randlabfa_{k})_{k}$, the iid sequence $(U_{k})_{k}$, and $\randfa_{\infty}$ are independent
by construction and
$$
	((\randlabfa_{k})_{k}, (U_{k})_{k}, \sigma(\randfa_{\infty}))
		\eqd ((\randlabfa_{k})_{k}, (U_{k})_{k}, \randfa_{\infty})
$$
since the
feature allocation is exchangeable and the independence used above still holds.
So
$$
	g((\randlabfa_{\tau(k)})_{k}, (U_{k})_{k}, \sigma(\randfa_{\infty}))
		\stackrel{d}{=} g((\randlabfa_{k})_{k}, (U_{k})_{k}, \randfa_{\infty})
$$
It follows that the sequence
$(\indfar_{n})_{n}$ is exchangeable.
\end{proof}
 
We can recover the full feature allocation $\randfa_{\infty}$ from
the sequence $\indfar_{1},\indfar_{2},\ldots$.
In particular, if $\{x_{1}, x_{2}, \ldots\}$ are the unique values in
$\{\indfar_{1}, \indfar_{2}, \ldots\}$, then the features are
$\{\{n: x_{k} \in \indfar_{n}\}: k = 1,2,\ldots\}$. The feature allocation
can similarly be recovered from the order-of-appearance label collections
$(\indfao_{n})$.

We can also recover a new \emph{random ordered feature allocation}
$\labrandfa_{N}$ from the sequence $(\indfar_{n})$. In particular,
$\labrandfa_{N}$ is the sequence---rather than the collection---of
features $\{n: x_{k} \in \indfar_{n}\}$ such that the feature with 
smallest label $\randlabfa_{k}$ occurs first, and so on. This 
construction achieves our goal of avoiding the combinatorial
factors needed to work with the distribution of $\randfa_{N}$,
while retaining exchangeability and consistency.

%%%
\begin{figure}
	%\centerline{
		\begin{minipage}[t]{0.10\textwidth}
			\vspace{0pt}
			\scalebox{0.5}{
				\input{mx_feat2.pstex_t}
			}
		\end{minipage}
		\begin{minipage}[t]{0.18\textwidth}
			\vspace{11pt}
			\scalebox{0.5}{
				\input{rmx_feat3.pstex_t}
			}
		\end{minipage}
		\begin{minipage}[t]{0.23\textwidth}
			\vspace{11pt}
			\scalebox{0.5}{
				\input{rmx_feat4.pstex_t}
			}
		\end{minipage}
		\begin{minipage}[t]{0.23\textwidth}
			\vspace{11pt}
			\scalebox{0.5}{
				\input{rmx_feat5.pstex_t}
			}
		\end{minipage}
		\begin{minipage}[t]{0.23\textwidth}
			\vspace{11pt}
			\scalebox{0.5}{
				\input{rmx_feat6.pstex_t}
			}
		\end{minipage}
	%}
\caption{\label{fig:rmx_feat} The same consistent sequence of feature allocations in \fig{mx_feat} but now with the uniform random order of \ex{feat_6_unif_rand_lab_spec} instead of the order of appearance illustrated in \fig{mx_feat}.
}
\end{figure}
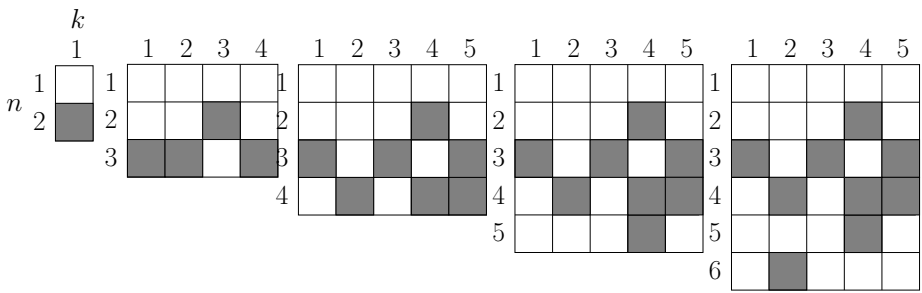
%%%

%~~~~~~~~~~~~~~~~~~~~~~~~~~~~~~~~~~~~~~~~~~
\begin{example}[Feature labeling schemes (continued)] \label{ex:feat_6_unif_rand_lab_spec}
Once more, consider the feature allocation
$$
	\detfa_{6} = \{\{2,5,4\},\{3,4\},\{6,4\},\{3\},\{3\}\}.
$$
and the uniform random labeling in \eq{feat_6_unif_rand_lab}.
If it happens that $\randlabfa_{4} < \randlabfa_{5} < \randlabfa_{2} < \randlabfa_{1} < \randlabfa_{3}$, then
the random ordered feature allocation is
$$
	\labdetfa_{6} = (\{3\},\{6,4\},\{3\},\{2,5,4\},\{3,4\}).
$$
\end{example}
%~~~~~~~~~~~~~~~~~~~~~~~~~~~~~~~~~~~~~~~~~~

Recall that we were motivated by \ex{two_bern} to produce such a random
ordering scheme to avoid obfuscating combinatorial factors in the probability
of a feature allocation. From another
perspective, these factors arise because the random labeling is in some sense
more natural than alternative labelings; again, consider random labels
as iid parameters for each feature.
While order-of-appearance labeling is common due to its
pleasant aesthetic representation in matrix form (compare \figs{mx_feat} and \figss{rmx_feat}), one must be careful to remember
that the resulting label sets $(\indfao_{n})$ are not exchangeable. We will
use random labeling extensively below since, among other nice properties,
it preserves exchangeability of the sets of feature labels associated
with the indices.
	
%%%%%%%%%%%%%%%%%%%%%%%%%%%%
\section{Exchangeable feature probability function} \label{sec:efpf}
%%%%%%%%%%%%%%%%%%%%%%%%%%%%

In general,
given a probability of a random feature allocation, $\mbp(\randfa_{N} = \detfa_{N})$,
we can find the probability of a random ordered feature allocation
$\mbp(\labrandfa_{N} = \labdetfa_{N})$ as follows. Let $H$ be the 
number of distinct features of $\randfa_{N}$, and let $(\ok_{1},\ldots,\ok_{H})$
be the multiplicities of these distinct features in decreasing order. Then
\begin{equation}
	\label{eq:order_mult}
	\mbp(\labrandfa_{N} = \labdetfa_{N})
		= \binom{ K }{ \ok_{1}, \ldots, \ok_{H} }^{-1} \mbp(\randfa_{N} = \detfa_{N}),
\end{equation}
where
$$
	\binom{ K }{ \ok_{1}, \ldots, \ok_{H} } := \frac{ K! }{ \ok_{1}! \cdots \ok_{H}!}.
$$

For partitions, the effect of this multiplicative factor is the same across all
partitions with the same number of clusters; for some number of clusters $K$, it is just
$1 / K!$. In the general feature case, the multiplicative factor may be different for 
different feature configurations with the same number of features.

%~~~~~~~~~~~~~~~~~~~~~~~~~~~~~~~~~~~~~~~~~~
\begin{example}[A Bernoulli, two-feature allocation (continued)]

Consider $\randfa_{N}$ constructed
as in \ex{two_bern}. Denote the sizes of the two features by $M_{N,1}$ and $M_{N,2}$. Then
\begin{align}
	\nonumber
	\mbp(\labrandfa_{N} = \labdetfa_{N})
		&= \frac{1}{2} q_{A}^{M_{N,1}} (1-q_{A})^{N - M_{N,1}} q_{B}^{M_{N,2}} (1-q_{B})^{N - M_{N,2}} \\
		\nonumber
		& {} + \frac{1}{2} q_{A}^{M_{N,2}} (1-q_{A})^{N - M_{N,2}} q_{B}^{M_{N,1}} (1-q_{B})^{N - M_{N,1}} \\
		\label{eq:efpf_two_bern}
		&= \efpf(N, M_{N,1}, M_{N,2}).
\end{align}
Here, $\efpf$ is some function of the number of indices $N$
and the feature sizes $(M_{N,1}, M_{N,2})$ that we
note is symmetric in $(M_{N,1}, M_{N,2})$;
i.e., $\efpf(N,M_{N,1},M_{N,2}) = \efpf(N,M_{N,2},M_{N,1})$.
\end{example}
%~~~~~~~~~~~~~~~~~~~~~~~~~~~~~~~~~~~~~~~~~~

When the feature allocation probability admits the representation
\begin{equation}
	\label{eq:efpf}
	\mbp(\labrandfa_{N} = \labdetfa_{N}) = \efpf(N,|\blockfa_{1}|, \ldots, |\blockfa_{K}|)
\end{equation}
for every ordered feature allocation
$
	\labdetfa_{N} = (A_{1},\ldots,A_{K})
$
and some function $\efpf$ that is symmetric in all arguments after the first,
we call $\efpf$ the \emph{exchangeable feature probability function} (EFPF). 
We take care to note that the exchangeable partition probability function (EPPF), which
always exists for partitions, is not a special case of the EFPF. Indeed, the EPPF assigns zero
probability to any multiset in which an index occurs in more than one feature of the
multiset; e.g., $\{\{1\},\{2\}\}$ is a valid partition and a valid feature allocation of $[2]$,
but $\{\{1\},\{1\}\}$ is a valid feature allocation but not a valid partition of $[2]$.
Thus, the EPPF must examine the feature indices of a feature allocation to judge their
exclusivity and thereby assign a probability.
By contrast, the indices in the multiset provide no such information to the EFPF;
only the sizes of the multiset features are relevant in the EFPF case.

%%%
\begin{proposition}
\label{prop:efpf}
The class of exchangeable feature allocations with EFPFs is a strict but non-empty subclass
of the class of exchangeable feature allocations.
\end{proposition}
%%%

\begin{proof}
\ex{epf_ibp} below shows that the class of feature allocations with EFPFs 
is non-empty, and \ex{two_feat} below establishes that there exist simple 
exchangeable feature allocations without EFPFs.
\end{proof}

%~~~~~~~~~~~~~~~~~~~~~~~~~~~~~~~~~~~~~~~~~~
\begin{example}[Three-parameter Indian buffet process] \label{ex:epf_ibp}
The Indian buffet process (IBP)~\citep{griffiths:2006:infinite}
is a generative model for a random feature allocation that is 
specified recursively in a manner akin to the Chinese restaurant 
process~\citep{aldous:1985:exchangeability} in the case of partitions. 
The metaphor involves a set of ``customers'' that enter a restaurant
and sample a set of ``dishes.'' Order the customers by placing them 
in one-to-one correspondence with the indices $n \in \mathbb{N}$. 
The dishes in the restaurant correspond to feature labels. 
Customers in the Indian buffet can sample any non-negative 
integer number of dishes. The set of dishes chosen by a customer 
$n$ is just $\indfao_{n}$, the collection of feature labels for 
the features to which $n$ belongs, and the procedure described 
below provides a way to construct $\indfao_{n}$ recursively.

%%%
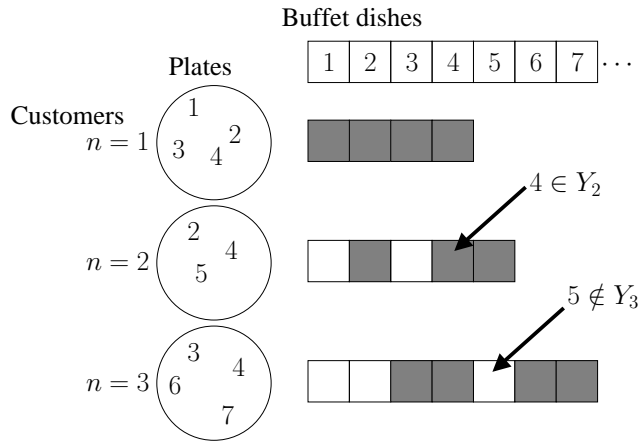
\begin{figure}[t]
	\centerline{
		\scalebox{0.5}{
			\input{ibp.pstex_t}
		}
	}
	\caption{\label{fig:ibp} Illustration of an Indian buffet process in the order-of-appearance representation of \fig{mx_feat}. The buffet (\emph{top}) consists of a vector of dishes, corresponding to features. Each customer---corresponding to a data point---who enters the restaurant first decides whether or not to choose dishes that the other customers have already sampled.  The customer then selects a random number of new dishes, not previously sampled by any customer. A gray box in position $(n,k)$ indicates customer $n$ has sampled dish $k$, and a white box indicates the customer has not sampled the dish. In the example, the second customer has sampled exactly those dishes indexed by 2, 4, and 5: $\indfao_{2} = \{2,4,5\}$.}
\end{figure}
%%%

We describe an extended version \citep{teh:2009:indian,broderick:2012:beta}
of the Indian buffet that includes two extra parameters beyond the single 
\emph{mass parameter} $\bpmass$ ($\bpmass > 0$) originally specified by 
\citet{griffiths:2006:infinite}; in particular, we include a \emph{concentration 
parameter} $\bpconc$ ($\bpconc > 0$) and a \emph{discount parameter} $\bpdisc$ 
($\bpdisc \in [0,1)$).  We abbreviate this three-parameter IBP as ``3IBP.'' 
The single-parameter IBP may be recovered by setting $\bpconc = 1$ and $\alpha = 0$.

We start with a single customer, who enters the buffet and chooses $K^{+}_{1} \sim \pois(\bpmass)$ dishes. None of the dishes have been sampled by any other customers since no other customers have yet entered the restaurant. An order-of-appearance labeling gives the dishes labels $1,\ldots,K^{+}_{1}$ if $K^{+}_{1} > 0$.

Recursively, the $n$th customer chooses which dishes to sample in two phases. First, for each dish $k$ that has previously been sampled by any customer in $1,\ldots,n-1$, customer $n$ samples dish $k$ with probability
$$
	\frac{M_{n-1,k} - \bpdisc}{\bpconc + n - 1},
$$
for $M_{n,k}$ equal to the number of customers indexed $1,\ldots,n$ who have tried dish $k$. As each dish represents a feature, sampling a dish represents that the customer index $n$ belongs to that feature. And $M_{n,k}$ is the size of the feature labeled $k$ in the feature allocation of $[n]$. 

Next, customer $n$ chooses
$$
	K^{+}_{n} \sim \pois\left( \bpmass \frac{\Gamma(\bpconc + 1)}{\Gamma(\bpconc + n)}
		\cdot \frac{\Gamma(\bpconc + \bpdisc - 1 + n)}{\Gamma(\bpconc + \bpdisc)} \right)
$$
new dishes to try. If $K^{+}_{n} > 0$, then the dishes receive unique order-of-appearance labels $K_{n-1}+1, \ldots, K_{n}$. Here, $K_{n}$ represents the number of sampled dishes after $n$ customers: $K_{n} = K_{n-1} + K^{+}_{n}$ (with base case $K_{0} = 0$).

With this generative model in hand, we can find the probability of a particular feature allocation. We discover its form by enumeration. At each round $n$, we have a Poisson number of new features, $K^{+}_{n}$, represented. The probability factor associated with these choices is a product of Poisson densities:
$$
	\prod_{n=1}^{N} \frac{1}{K^{+}_{n}!} [C(n,\bpmass,\bpconc,\bpdisc)]^{K^{+}_{n}} \exp\left( -C(n,\bpmass,\bpconc,\bpdisc) \right),
$$
where
$$
	C(n,\bpmass,\bpconc,\bpdisc) := \bpmass \frac{\Gamma(\bpconc + 1)}{\Gamma(\bpconc + n)}
		\cdot \frac{\Gamma(\bpconc + \bpdisc - 1 + n)}{\Gamma(\bpconc + \bpdisc)}.
$$

Let $R_{k}$ be the round on which the $k$th dish, in order of appearance, is first chosen. Then the denominators for future dish choice probabilities are the factors in the product $(\bpconc + R_{k}) \cdot (\bpconc + R_{k} + 1) \cdots (\bpconc + N - 1)$. The numerators for the times when the dish is chosen are the factors in the product $(1 - \bpdisc) \cdot (2 - \bpdisc) \cdots (M_{N,k}-1 - \bpdisc)$. The numerators for the times when the dish is not chosen yield $(\bpconc + R_{k} - 1 + \bpdisc) \cdots (\bpconc + N - 1 - M_{N,k} + \bpdisc)$.
Let $\blockfa_{n,k}$ represent the collection of indices in the feature with label $k$ after $n$ customers have entered the restaurant. Then $M_{n,k} = |\blockfa_{n,k}|$.

Finally, let $\ok_{1},\ldots,\ok_{H}$ be the multiplicities of distinct features formed by this model. We note that there are
$$
	\left[ \prod_{n=1}^{N} K^{+}_{n}! \right] / \left[ \prod_{h=1}^{H} \ok_{h}! \right]
$$
rearrangements of the features generated by this process that all yield the same feature allocation. Since they all have the same generating probability, we simply multiply by this factor to find the feature allocation probability.

Multiplying all factors together\footnote{Readers curious about how the $R_{k}$ terms disappear
may observe that
$$
	\prod_{k=1}^{K_{N}} \frac{\Gamma(\bpconc + R_{k})}{\Gamma(\bpconc + R_{k} + \bpdisc - 1)}
		= \prod_{n=1}^{N} \left( \frac{\Gamma(\bpconc + n)}{\Gamma(\bpconc + n + \bpdisc - 1)} \right)^{K_{N}^{+}}.
$$
}
and taking $\detfa_{n} = \{\blockfa_{N,1}, \ldots, \blockfa_{N,K_{N}}\}$ yields
\begin{align}
	\nonumber
	\lefteqn{ \mbp(\randfa_{N} = \detfa_{N}) } \\
		\nonumber
		&= \left( \prod_{h=1}^{H} \ok_{h}! \right)^{-1}
			\left( \bpmass \frac{\Gamma(\bpconc + 1)}{\Gamma(\bpconc + \bpdisc)} \right)^{K_{N}} 
			\exp\left( -\sum_{n=1}^{N} \bpmass \frac{\Gamma(\bpconc + 1)}{\Gamma(\bpconc + n)}
			\cdot \frac{\Gamma(\bpconc + \bpdisc - 1 + n)}{\Gamma(\bpconc + \bpdisc)} \right) \\
		\nonumber
		& \quad ~ \cdot \left[ \prod_{k=1}^{K_{N}} 
				\frac{\Gamma(M_{N,k} - \bpdisc)}{\Gamma(1-\bpdisc)}
				\cdot \frac{\Gamma(\bpconc + N - M_{N,k} + \bpdisc)}{\Gamma(\bpconc + N)}
			\right].
\end{align}

It follows from \eq{order_mult} that the probability of a uniform random ordering of the feature allocation is
\begin{align}
	\nonumber
	\lefteqn{ \mbp(\labrandfa_{N} = \labdetfa_{N}) } \\ %p(N,|\blockfa_{N,1}|,\ldots,|\blockfa_{N,K(N)}|)
		\nonumber
		&= \frac{1}{K_{N}!} \left( \bpmass \frac{\Gamma(\bpconc + 1)}{\Gamma(\bpconc + \bpdisc)} \right)^{K_{N}} 
			\exp\left( -\sum_{n=1}^{N} \bpmass \frac{\Gamma(\bpconc + 1)}{\Gamma(\bpconc + n)}
			\cdot \frac{\Gamma(\bpconc + \bpdisc - 1 + n)}{\Gamma(\bpconc + \bpdisc)} \right) \\
		\label{eq:ibp_efpf}
		& \quad ~ \cdot \left[ \prod_{k=1}^{K_{N}} 
				\frac{\Gamma(M_{N,k} - \bpdisc)}{\Gamma(1-\bpdisc)}
				\cdot \frac{\Gamma(\bpconc + N - M_{N,k} + \bpdisc)}{\Gamma(\bpconc + N)}
			\right].
\end{align}

The distribution of $\labrandfa_{N}$ has no dependence on the ordering of the indices in $[N]$. Hence,
the distribution of $\randfa_{N}$ depends only on the same quantities---the number of indices and the feature sizes---and the feature multiplicities. So we see that the 3IBP construction yields an exchangeable random feature allocation. Consistency follows from the recursive construction and exchangeability. Therefore, \eq{ibp_efpf} is seen to be in EFPF form given by \eq{efpf}.
\end{example}
%~~~~~~~~~~~~~~~~~~~~~~~~~~~~~~~~~~~~~~~~~~

The three-parameter Indian buffet process has an EFPF representation, but the following simple model does not.

%~~~~~~~~~~~~~~~~~~~~~~~~~~~~~~~~~~~~~~~~~~
\begin{example}[A general two-feature allocation]
\label{ex:two_feat}
We here describe an exchangeable, consistent random feature allocation
whose (ordered) distribution does not depend only on the number of indices $N$
and the sizes of the features of the
allocation.

Let $p_{10}, p_{01}, p_{11}, p_{00}$ be fixed frequencies that sum to one. 
Let $\indfaq_{n}$ represent the collection of features to which index $n$ belongs.
For $n \in \{1,2\}$, choose $\indfaq_{n}$ independently and identically
according to:
$$
	\indfaq_{n} = \left\{ \begin{array}{ll}
			\{1\} & \textrm{with probability } p_{10} \\
			\{2\} & \textrm{with probability } p_{01} \\
			\{1,2\} & \textrm{with probability } p_{11} \\
			\emptyset & \textrm{with probability } p_{00}.
		\end{array} \right.
$$
We form a feature allocation from these labels as follows. For each label ($1$ or $2$),
collect those indices $n$ with the given label appearing in $\indfaq_{n}$ to form a feature.

Now consider two possible outcome feature allocations:
$\detfa_{2} = \{\{2\}, \{2\}\}$, and $\detfa'_{2} = \{\{1\},\{2\}\}$.
The probability of any ordering $\labdetfa_{2}$ of $\detfa_{2}$ under this model is
$$
	\mbp(\labrandfa_{2} = \labdetfa_{2})
		= p_{10}^{0} \; p_{01}^{0} \; p_{11}^{1} \; p_{00}^{1}.
$$
The probability of any ordering $\labdetfa'_{2}$ of $\detfa'_{2}$ is
$$
	\mbp(\labrandfa_{2} = \labdetfa'_{2})
		= p_{10}^{1} \; p_{01}^{1} \; p_{11}^{0} \; p_{00}^{0}.
$$
It follows from these
two probabilities that we can choose values of $p_{10},p_{01},p_{11},p_{00}$ such
that $\mbp(\labrandfa_{2} = \labdetfa_{2}) \ne \mbp(\labrandfa_{2} = \labdetfa'_{2})$. But 
$\labdetfa_{2}$ and $\labdetfa'_{2}$ have the same feature counts and $N$
value ($N = 2$). So there can be no such symmetric function $\eppf$,
as in \eq{efpf_two_bern}, for this model.
\end{example}
%~~~~~~~~~~~~~~~~~~~~~~~~~~~~~~~~~~~~~~~~~~

%%%%%%%%%%%%%%%%%%%%%%%%%%%%
\section{The Kingman paintbox and feature paintbox} \label{sec:paintbox}
%%%%%%%%%%%%%%%%%%%%%%%%%%%%

Since the class of exchangeable feature models with EFPFs is a strict subclass of
the class of exchangeable feature models, it remains to find a characterization
of the latter class. Noting that the sequence of feature collections $\indfar_{n}$ 
is an exchangeable sequence when the uniform random labeling of features is used,
we might turn to the de Finetti mixing measure of this exchangeable sequence for
such a characterization.

Indeed, in the partition case, the Kingman paintbox
\citep{kingman:1978:representation,aldous:1985:exchangeability}
provides just such a characterization.
%%%
\begin{theorem}[Kingman paintbox]
\label{thm:paintbox}
Let $\randpart_{\infty} := (\randpart_{n})_{n=1}^{\infty}$ be an exchangeable random partition of $\mathbb{N}$, and let $(\cfreqord_{n,k}, k \ge 1)$ be the decreasing rearrangement of cluster sizes of $\randpart_{n}$ with $\cfreqord_{n,k} = 0$ if $\randpart_{n}$ has fewer than $k$ clusters. Then $\cfreqord_{n,k}/n$ has an almost sure limit $\pfreqord_{k}$ as $n \rightarrow \infty$ for each $k$. Moreover, the conditional distribution of $\randpart_{\infty}$ given $(\pfreqord_{k}, k \ge 1)$ is as if $\randpart_{\infty}$ were generated by random sampling from a random distribution with ranked atoms $(\pfreqord_{k}, k \ge 1)$.
\end{theorem}
%%%

%%%
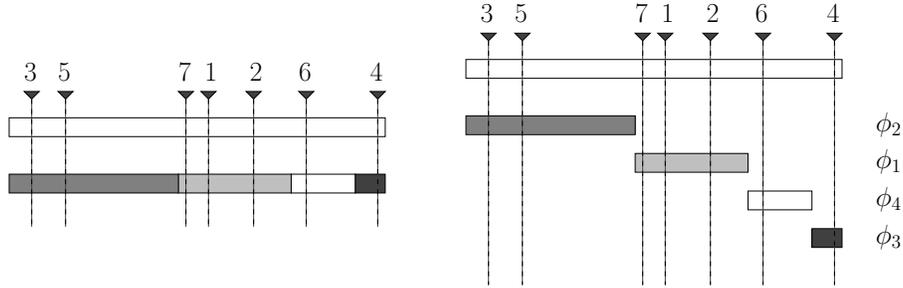
\begin{figure}
	\begin{minipage}{0.45\textwidth}
		\centerline{
			\scalebox{0.5}{
				\input{kingman.pstex_t}
			}
		}
	\end{minipage}
	\hspace{0.05\textwidth}
	\begin{minipage}{0.45\textwidth}
		\centerline{
			\scalebox{0.5}{
				\input{kingman_levels.pstex_t}
			}
		}
	\end{minipage}
\caption{\label{fig:kingman}
\emph{Left}: An example Kingman paintbox. The upper rectangle
represents the unit interval. The lower rectangles represent a partition of the
unit interval into four subintervals corresponding to four clusters. The horizontal locations
of the seven vertical lines represent seven uniform random draws from the unit interval.
The resulting partition of $[7]$ is $\{\{3,5\}, \{7,1,2\}, \{6\}, \{4\}\}$. \emph{Right}:
An alternate representation of the same Kingman paintbox, now with each subinterval
separated out into its own vertical level. To the right of each cluster subinterval is a
uniform random label (with index determined by order of appearance) for the cluster. 
}
\end{figure}
%%%

When the partition clusters are labeled with uniform random labels
rather than by the ranking in the statement of the theorem above,
Kingman's paintbox provides the de Finetti mixing measure for the sequence of
partition labels of each index $n$. Two representations of an example Kingman
paintbox are illustrated in \fig{kingman}. The Kingman paintbox is so named
since we imagine each subinterval of the unit interval as containing paint of 
a certain color; the colors have a one-to-one mapping with the uniform random
cluster labels.
A random draw from the unit interval is painted with the color of 
the Kingman paintbox subinterval into which it falls. While \fig{kingman}
depicts just four subintervals and hence at most four clusters,
the Kingman paintbox may in general have a countable number
of subintervals and hence clusters. Moreover, these subintervals
may themselves be random.

Note that the ranked atoms need not sum to one; in general, $\sum_{k} \pfreqord_{k} \le 1$.
When random sampling from the Kingman paintbox does not select some atom $k$
with $\pfreqord_{k} > 0$,
a new cluster is formed but it is necessarily never selected again for another index.
In particular, then, a corollary of the Kingman paintbox theorem is that there are two types of clusters:
those with unbounded size as the number of indices $N$ grows to infinity and those
with exactly one member as $N$ grows to infinity; the latter are sometimes referred 
to as \emph{singletons} or collectively as \emph{Kingman dust}.
In the feature case, we impose one further regularity condition that essentially
rules out dust.
Consider any feature allocation $\randfa_{\infty}$.
Recall that we use the notation $\indfar_{n}$ to indicate the set of features to which
index $n$ belongs. We assume that, for each $n$, with probability one there
exists some $m$ with $m \ne n$ such that $\indfar_{m} = \indfar_{n}$. Equivalently, with probability
one there is no index with a unique feature collection.
We call a random feature allocation that obeys this condition a
\emph{regular feature allocation}.

We can prove the following theorem for the feature case,
analogous to the Kingman paintbox construction for partitions.

%%%
\begin{theorem}[Feature paintbox]
\label{thm:feature_paintbox}
Let $\randfa_{\infty} := (\randfa_{n})$ be an exchangeable, consistent, regular
random feature allocation of $\mathbb{N}$.
There exists a random sequence $(C_{k})_{k=1}^{\infty}$ 
such that $C_{k}$ is a countable union of subintervals of $[0,1]$ (and may be 
empty) 
and such that
$\randfa_{\infty}$ has the same distribution as $\randfa'_{\infty}$
where $\randfa'_{\infty}$ is
generated as follows.
Randomly sample $(U'_{n})_{n}$ iid uniform in $[0,1]$. 
Let $\indfaq_{n} := \{k: U'_{n} \in C_{k}\}$ represent a collection of feature
labels for index $n$, and let $\randfa'_{\infty}$ be the induced 
feature allocation from these label collections.
\end{theorem}
%%%

\begin{proof}
Given $\randfa_{\infty}$ as in the theorem statement, we can construct
$(\indfar_{n})_{n=1}^{\infty}$ as in \lem{unif_rand_exch}. Then, according to
\lem{unif_rand_exch},
$(\indfar_{n})_{n=1}^{\infty}$ is an exchangeable sequence.
Note that $\indfar_{n}$ defines a partition:
$n \sim m$ (i.e., $n$ and $m$ belong to the same cluster of the partition)
if and only if $\indfar_{n} = \indfar_{m}$.
This partition is exchangeable since the feature allocation is.
Moreover, since we assume there are no singletons in the induced
partition (by regularity), the Kingman paintbox theorem implies that the
Kingman paintbox atoms sum to one.

By de Finetti's theorem \citep{aldous:1985:exchangeability},
there exists $\alpha$ such that $\alpha$ is the directing random measure for $(\indfar_{n})$. Condition on $\alpha = \mu$. Write $\mu = \sum_{j=1}^{\infty} q_{j} \delta_{x_{j}}$, where the $q_{j}$ satisfy $q_{j} \in (0,1]$ and are written in monotone decreasing order: $q_{1} \ge q_{2} \ge \cdots$. The condition that the atoms of the paintbox sum to one translates to $\sum_{j=1}^{\infty} q_{j} = 1$. The $(x_{j})$ are the (countable) unique values of $\indfar_{n}$, ordered to agree with the $q_{j}$.
The strong law of large numbers yields
$$
	N^{-1} \#\{n: n \le N, \indfar_{n} = x_{j}\} \rightarrow q_{j}, \quad N \rightarrow \infty.
$$

Since $\sum_{j=1}^{\infty} q_{j} = 1$, we can partition the unit interval into subintervals of length $q_{j}$. The $j$th
such subinterval starts at $s_{j} := \sum_{l=1}^{j-1} q_{l}$ and ends at
$e_{j} := s_{j+1}$.
For $k = 1,2,\ldots$, define $C_{k} := \bigcup_{j: \randlabfa_{k} \in x_{j}} [s_{j},e_{j})$.
We call the $(C_{k})_{k=1}^{\infty}$ the \emph{feature paintbox}.

Then $\randfa_{\infty}$ has the same
distribution as the following construction. Let $(U'_{1},U'_{2},\ldots)$ be
an iid sequence of uniform random variables. For each $n$,
define $\indfaq_{n} = \{k: U'_{n} \in C_{k}\}$ to be the collection of 
features, now labeled by positive integers, to which $n$ belongs. Let $\randfa'_{\infty}$
be the feature allocation induced by the $(\indfaq_{n})$.
\end{proof}

A point to note about this feature paintbox construction is that
the ordering of the feature paintbox subsets $C_{k}$ in the proof is given by the
order of appearance of features in the original feature allocation $\randfa_{\infty}$.
This ordering stands in contrast to the ordering of atoms by size
in the Kingman paintbox. Making use of such a size-ordering would be more difficult in
the feature case due to the non-trivial intersections of feature subsets. A particularly
important implication is that the conditional distribution of $\randfa_{\infty}$ 
given $(C_{k})_{k}$ is not the same as that of $\randfa'_{\infty}$ given 
$(C_{k})_{k}$ (cf.\ \citet{pitman:1995:exchangeable} for
similar ordering issues in the partition case).

%%%
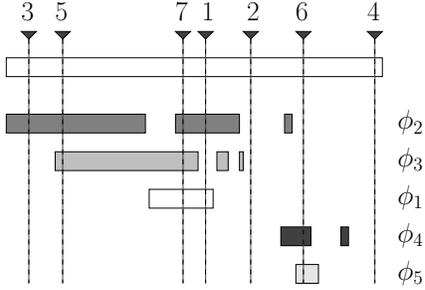
\begin{figure}
	\centerline{
		\scalebox{0.5}{
			\input{feat_paint.pstex_t}
		}
	}
\caption{\label{fig:feat_paint}
An example feature paintbox. The top rectangle
represents the unit interval. Each vertical level below the top rectangle
represents a subset of the unit interval corresponding to a feature.
To the right of each subset is a uniform random label for the feature.
For example, using the notation of \thm{feature_paintbox},
the topmost subset is $C_{2}$ corresponding to feature label $\randlabfa_{2}$.
The vertical dashed lines represent uniform random draws; i.e., $U'_{n}$
for index $n$.
The resulting feature allocation of $[7]$ for this realization of the construction is
$\{\{3,5,7,1\}, \{5,7\}, \{7,1\}, \{6\}, \{6\}\}$. The collection of feature labels
for index $7$ is $\indfaq_{7} = \{\randlabfa_{2}, \randlabfa_{3}, \randlabfa_{1}\}$.
The collection of feature labels for index $4$ is $\indfaq_{4} = \emptyset$.
}
\end{figure}
%%%

An example feature paintbox is illustrated in \fig{feat_paint}. Again,
we may think of each feature paintbox subset as containing paint of a certain color
(where these colors have a one-to-one mapping with the uniform
random labels). Draws from the unit interval to determine the feature allocation may now be painted with
some subset of these colors rather than just a single color.

Next, we revisit earlier examples to find their feature paintbox representations.

%%%
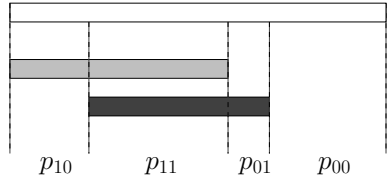
\begin{figure}
	\centerline{
		\scalebox{0.5}{
			\input{two_feat_paint2.pstex_t}
		}
	}
\caption{\label{fig:two_feat_paint}
A feature paintbox for the two-feature allocation
in \ex{two_feat}. The top rectangle is the unit interval. The 
middle rectangle is the feature paintbox subset for feature 1.
The lower rectangle is the feature paintbox subset for feature 2.}
\end{figure}
%%%

%~~~~~~~~~~~~~~~~~~~~~~~~~~~~~~~~~~~~~~~~~~
\begin{example}[A general two-feature allocation (continued)]
\label{ex:two_feat_paint}

The feature paintbox for the random feature allocation
in \ex{two_feat} consists of two features. The total measure of the 
paintbox subset for
feature 1 is $p_{10} + p_{11}$. The total measure of the paintbox
subset for feature 2 is $p_{01} + p_{11}$. The total measure of the intersection
of these two subsets is $p_{11}$. A depiction of this paintbox appears in
\fig{two_feat_paint}.
\end{example}
%~~~~~~~~~~~~~~~~~~~~~~~~~~~~~~~~~~~~~~~~~~

%~~~~~~~~~~~~~~~~~~~~~~~~~~~~~~~~~~~~~~~~~~
\begin{example}[Three-parameter Indian buffet process (continued)]
\label{ex:extend_ibp_bp_tmp}

The 3IBP turns out to be an instance of a general class of exchangeable 
feature models that we refer to as \emph{feature frequency models}.  This class 
of models not only provides a straightforward way to construct feature 
paintbox representations in general, but also plays a key role in our 
general theory, providing a link between feature paintboxes and EFPFs.
In the following section, we define feature frequency models, develop the general 
construction of paintboxes from feature frequency models, and then return to 
the construction of the feature paintbox for the 3IBP as an example.
We subsequently turn to the general theoretical characterization of feature frequency
models.
\end{example}
%~~~~~~~~~~~~~~~~~~~~~~~~~~~~~~~~~~~~~~~~~~

%%%%%%%%%%%%%%%%%%%%%%%%%%%%
\section{Feature frequency models} \label{sec:freq}
%%%%%%%%%%%%%%%%%%%%%%%%%%%%

We now discuss a general class of exchangeable feature models for which it is 
straightforward to describe the feature paintbox.
Let $(V_{k})$ be a sequence of (not necessarily independent) random variables
with values in $[0,1]$ such that $\sum_{k=1}^{\infty} V_{k} < \infty$ almost surely.
Let $\randlabfa_{k} \stackrel{iid}{\sim} \unif[0,1]$
and independent of the $(V_{k})$.
A \emph{feature frequency model}
is built around a random measure
$\bproc = \sum_{k=1}^{\infty} V_{k} \delta_{\randlabfa_{k}}$.
We may draw a feature allocation given $\bproc$ as follows. For each data point $n$,
independently draw its features like so: for each feature indexed by $k$, 
independently make a Bernoulli draw with success probability $V_{k}$. If
the draw is a success, $n$ belongs to the feature indexed by $k$ (i.e.,\ the feature with
label $\randlabfa_{k}$). If the draw is a failure, $n$ does not belong to
the feature indexed by $k$. The feature allocation is induced in the usual 
way from these labels.

The condition that the frequencies have an almost surely finite sum
guarantees, by the Borel-Cantelli lemma, that the number of features
exhibited by any index $n$ is almost surely finite, as required in the definition
of a feature allocation.  We obtain exchangeable feature allocations
simply by virtue of the fact that the feature allocations are independently
and identically distributed given $\bproc$.  The Bernoulli draws from the 
feature frequencies guarantee that the feature allocation is regular.

Before constructing the feature paintbox for such a model, we note
that $V_{k}$ is the total length of the paintbox subset for the feature
indexed by $k$. In this sense, it is the frequency of this feature (hence the name
``feature frequency model''). And $\randlabfa_{k}$ is the uniform random feature label
for the feature with frequency $V_{k}$.
Finally, to achieve the independent Bernoulli draws across $k$ required by the
feature allocation specification, we need for the intersection of any
two paintbox subsets to have length equal to the product of the two
paintbox subset lengths. This desideratum can be achieved with a recursive construction.

First, divide the unit interval into one subset (call it $I_{1}$) of length $V_{1}$
and another subset (call it $I_{0}$) of length $1 - V_{1}$. Then $I_{1}$
is the paintbox subset for the feature indexed by 1.
Recursively, suppose we have paintbox subsets for features indexed $1$ to $K-1$.
Let $e$ be a binary string of length $K-1$.
Suppose that $I_{e}$ is the intersection of (a)
all paintbox subsets for features indexed
by $k$ ($k < K$) where the $k$th digit of $e$ is 1 and (b) all paintbox subset complements
for features indexed by $k$ ($k < K$) where the $k$th digit of $e$ is 0.
For every $e$, we construct $I_{(e,1)}$ to be a subset of $I_{e}$ with
total length equal to $V_{K}$ times the length of $I_{e}$.
We construct $I_{(e,0)}$ to be $I_{e} \backslash I_{(e,1)}$.

%%%
\begin{figure}
	\centerline{
		\scalebox{0.5}{
			\input{freq_paint.pstex_t}
		}
	}
\caption{\label{fig:freq_paint} An example feature paintbox for
a feature frequency model (\mysec{freq}). One such model is the
3IBP (\ex{extend_ibp_bp}).
}
\end{figure}
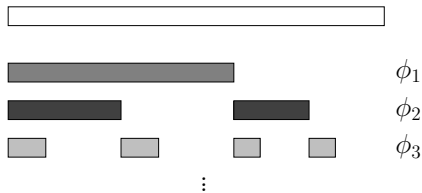
%%%

Finally, the paintbox subset for the feature indexed by $K$ is the union of
all $I_{e'}$ with $e'$ a binary string of length $K$ such that the final digit of $e'$
is 1. An example of such a paintbox is illustrated in \fig{freq_paint}.

%~~~~~~~~~~~~~~~~~~~~~~~~~~~~~~~~~~~~~~~~~~
\begin{example}[Three-parameter Indian buffet process (continued)]
\label{ex:extend_ibp_bp}

We show that the three-parameter Indian buffet process is an example
of a feature frequency model, and thus its feature paintbox can be constructed
according to the general recipe that we have just presented.

The underlying random measure for the three-parameter Indian buffet 
process is known as the \emph{three-parameter beta process} 
\citep{teh:2009:indian,broderick:2012:beta}.
This random measure, denoted $\bproc$, can be constructed
explicitly via the following recursion (with $K_{0} = 0$):
\begin{align*}
	K^{+}_{n} &\sim \pois\left( \bpmass \frac{\Gamma(\bpconc + 1)}{\Gamma(\bpconc + n)}
		\cdot \frac{\Gamma(\bpconc + \bpdisc - 1 + n)}{\Gamma(\bpconc + \bpdisc)} \right), \\
	K_{n} &= K_{n-1} + K_{n}^{+} \\
	V_{k} &\sim \tb(1 - \bpdisc, \bpconc + n + \bpdisc), \quad k = K_{n-1} + 1, \ldots, K_{n} \\
	\randlabfa_{k} &\sim \unif[0,1] \\
	\bproc &= \sum_{k=1}^{\infty} V_{k} \delta_{\randlabfa_{k}},
\end{align*}
where we recall that the $\randlabfa_{k}$ are assumed to be drawn from 
the uniform distribution for simplicity in this paper, but in general
they may be drawn from a continuous distribution that serves as a prior
for the parameters defining a likelihood.

Given $\bproc = \sum_{k=1}^{\infty} V_{k} \delta_{\randlabfa_{k}}$,
the feature allocation is drawn according to the procedure outlined 
for feature frequency models conditioned on the underlying random measure.
Building on work of~\citet{thibaux:2007:hierarchical} in the case
of the IBP, \citet{teh:2009:indian} demonstrate that the distribution 
of the resulting feature allocation is the same as if it were generated 
according to a three-parameter Indian buffet process.
\end{example}
%~~~~~~~~~~~~~~~~~~~~~~~~~~~~~~~~~~~~~~~~~~

We have seen that the 3IBP can be represented
as a feature frequency model. It is straightforward to observe that the two-feature
model in \exs{two_feat} and \exss{two_feat_paint}
cannot be represented as a feature frequency model unless the intersection 
of the feature subsets has length $p_{11}$ equal to the product 
of the feature subset lengths ($p_{10} + p_{11}$ and $p_{01} + p_{11}$);
i.e., unless $(p_{10} + p_{11}) (p_{01} + p_{11}) = p_{11}$ (cf.\ \fig{two_feat_paint}).
Therefore, we have the following result similar to \prop{efpf}.

%%%
\begin{proposition}
\label{prop:freq}
The class of feature frequency models is a strict but non-empty subclass
of the class of exchangeable feature allocations.
\end{proposition}
%%%

In proving \props{freq} and \propss{efpf}, we used the 3IBP 
as an example that belongs to both the class of feature models with EFPFs
and the class of feature frequency models. Moreover, in both cases we used two-feature
models as an example of exchangeable feature models that do not belong 
to these subclasses; in particular, we used two-feature models in which 
the feature combination probabilities $p_{10}, p_{01}, p_{11}, p_{00}$
are not in the necessary proportions. These observations suggest
that feature frequency models and EFPFs may be linked. We flesh out the relationship
between the two representations in the next few results. 

We start with
\emph{a priori labeled} features.
Recall from \mysec{label} that an a priori labeled feature allocation is to a feature allocation
what a classification is to a clustering; that is, the feature labels are known in advance.
The case where we know the feature
order in advance is somewhat easier and gives intuition for the type of result
we would like in the true feature allocation case.
In particular, we prove the results for the case of two a priori labeled features in
\thm{two_labeled_freq} and then the case of an unbounded number of a priori
labeled features in \thm{any_labeled_freq}.

From there, we move on to the (a priori) unlabeled case that is the focus of the paper
and prove the equivalence of EFPFs and a slight extension of feature frequency models in 
\thm{efpf_freq}.

%%%
\begin{theorem} \label{thm:two_labeled_freq}
Consider a model with two a priori labeled features: feature 1 and feature 2.
If the two features are generated from
labeled feature frequencies, the probability of an a priori labeled feature allocation
of $[N]$
with $M_{N,1}$ occurrences of feature 1 and $M_{N,2}$ occurrences
of feature 2 takes the form $\oefpf(N; M_{N,1}, M_{N,2})$, where we make no 
symmetry assumptions about $\oefpf$ here and also
allow any of $M_{N,1}$ and $M_{N,2}$ to be zero. Conversely, if the probability of any
a priori labeled feature allocation can be written as $\oefpf(N; M_{N,1}, M_{N,2})$, then
the feature allocation has the same distribution as if it were generated
from labeled feature frequencies.
\end{theorem}
%%%

\begin{proof}
Note that throughout this proof we consider the probability of
\emph{a particular} labeled feature allocation
of $[N]$ with $M_{N,1}$ occurrences of feature 1 and $M_{N,2}$ occurrences of feature 2,
as distinct from the probability of all labeled feature allocations
of $[N]$ with $M_{N,1}$ occurrences of feature 1 and $M_{N,2}$
occurrences of feature 2. The latter, which is
not addressed here, would be the sum over instances of the former. In particular,
recalling the matrix representation from \mysec{label}, there are 
$$
	\binom{N}{M_{N,1}} \binom{N}{M_{N,2}}
$$
possible $N \times 2$ matrices with $M_{N,1}$ ones in the first column
and $M_{N,2}$ ones in the second column.

The reader may feel there is some similarity in this setup to the
two-feature allocation of
\exs{two_feat} and \exss{two_feat_paint}.
We note that the quantities $p_{10}, p_{01}, p_{11}, p_{00}$---which
retain essentially the same meaning as in \fig{two_feat_paint}---may now be random
and that their order is pre-specified and non-random.

First, we calculate the probability of a certain labeled feature configuration
under this model. Let $M'_{n,10}$ be the number of indices in $[n]$
with feature 1 but not feature 2. Let $M'_{n, 01}$ be the number of indices in 
$[n]$ with feature 2 but not feature 1. Let $M'_{n,00}$ count the indices
with neither feature, and let $M'_{n,11}$ count the indices with both features.
Then
\begin{align}
	\label{eq:two_feat_label_prob}
	\mbp(\hat{F}_{N,1} = \hat{f}_{N,1}, \hat{F}_{N,2} = \hat{f}_{N,2})
		&= \mbe( p_{10}^{M'_{N,10}} p_{01}^{M'_{N,01}} p_{11}^{M'_{N,11}} p_{00}^{M'_{N,00}}).
\end{align}

Denote the total
probabilities of features 1 and 2 as, respectively, $q_{1} = p_{10} + p_{11}$
and $q_{2} = p_{01} + p_{11}$. Suppose that we have a feature frequency model. 
This assumption implies that
\begin{equation}
	\label{eq:two_freq_fact}
	p_{10} \eqas q_{1} (1-q_{2}), \quad p_{01} \eqas (1-q_{1}) q_{2},
	\quad p_{11} \eqas q_{1} q_{2}, \quad p_{00} \eqas (1-q_{1}) (1-q_{2}),
\end{equation}
where any one of the equalities in \eq{two_freq_fact} implies the others.
It follows that
\begin{align}
	\label{eq:two_feat_label_efpf}
	\mbp(\hat{F}_{N,1} = \hat{f}_{N,1}, \hat{F}_{N,2} = \hat{f}_{N,2})
		&= \mbe[ q_{1}^{M_{N,1}} (1-q_{1})^{N-M_{N,1}} q_{2}^{M_{N,2}} (1-q_{2})^{N - M_{N,2}} ],
\end{align}
where $M_{n,1} = M'_{n,10} + M'_{n,11}$ is the total number of 
indices with feature 1, and likewise $M_{n,2} = M'_{n,01} + M'_{n,11}$ is
the total number of indices with feature 2.

So we see that making a feature frequency model assumption yields a
feature allocation probability in \eq{two_feat_label_efpf} that depends
only on $N, M_{N,1}, M_{N,2}$. Since we retain the known labeling in this example,
the probability is 
not symmetric in $M_{N,1}$ and $M_{N,2}$.

In the other direction, suppose we know that
\begin{align}
	\label{eq:label_efpf_assume}
	\mbp(\hat{F}_{N,1} = \hat{f}_{N,1}, \hat{F}_{N,2} = \hat{f}_{N,2})
		&= \oefpf(N, M_{N,1}, M_{N,2})
\end{align}
for some function $\oefpf$. Again, we make no symmetry assumptions about $\oefpf$ here,
and any of $M_{N,1}$ and $M_{N,2}$ may be zero.
Then frequencies $p_{10}, p_{01}, p_{11}, p_{00}$ must exist by the law of large numbers;
we note they may be random.

The assumption in \eq{label_efpf_assume} implies that
the configurations
\begin{align*}
	(M'_{4,10},M'_{4,01},M'_{4,00},M'_{4,11})
		&= (2,2,0,0) \\
	(M'_{4,10},M'_{4,01},M'_{4,00},M'_{4,11})
		&= (0,0,2,2) \\
	(M'_{4,10},M'_{4,01},M'_{4,00},M'_{4,11})
		&= (1,1,1,1)
\end{align*}
have the
same probability. That is, by \eq{two_feat_label_prob},
\begin{align*}
	\mbe[ p_{10}^{2} p_{01}^{2} ]
		= \mbe[ p_{11}^{2} p_{00}^{2} ]
		= \mbe[ p_{10} p_{01} p_{11} p_{00} ].
\end{align*}
It follows that
\begin{align*}
	\mbe[ (p_{10} p_{01} - p_{11} p_{00})^{2} ]
		= \mbe[ p_{10}^{2} p_{01}^{2} + p_{11}^{2} p_{00}^{2} - 2  p_{10} p_{01} p_{11} p_{00} ]
		= 0.
\end{align*}
So it must be that $p_{10} p_{01} \stackrel{a.s.}{=} p_{11} p_{00}$. 
Recall that this condition is familiar from \ex{two_feat}.

Adding $p_{10} p_{11}$ to both sides of the almost sure equality
and then further adding
$p_{11} (p_{01} + p_{11})$ to both sides yields
$$
	(p_{10} + p_{11}) (p_{01} + p_{11}) \eqas p_{11} (p_{10} + p_{01} + p_{11} + p_{00}),
$$
which reduces to
$$
	q_{1} q_{2} \eqas p_{11}
$$
from the definitions of $q_{1}$ and $q_{2}$ and from the fact that
$p_{10} + p_{01} + p_{11} + p_{00} = 1$.

By \eq{two_freq_fact} and surrounding text,
we see that \eq{label_efpf_assume} implies our model is a
feature frequency model.
Thus, the equivalence between models with a priori labeled EFPFs and
a priori labeled feature frequency models
in the case of two features results from simple algebraic manipulations.
\end{proof}

Extending the argument above becomes more tedious
when more than two features are involved. In the case of multiple, or
even countably many, labeled features, a more elegant proof exists.

%%%
\begin{theorem} \label{thm:any_labeled_freq}
Consider a model with features a priori labeled $1,2,3,\ldots$.
If the features are generated from
labeled feature frequencies, the probability of an a priori labeled feature allocation
of $[N]$ with $K$ or fewer features and $M_{N,k}$ occurrences of feature $k$ for $k \in \{1,\ldots,K\}$
takes the form $\oefpf(N; M_{N,1}, \ldots, M_{N,K})$, where we make no 
symmetry assumptions about $\oefpf$ here and note that
any of $M_{N,1}, \ldots, M_{N,K}$ may be zero. Call $\oefpf$ a \emph{labeled EFPF}. Conversely, if the probability of any a priori
labeled feature allocation can be written as $\oefpf(N; M_{N,1}, \ldots, M_{N,K})$, then
the feature allocation has the same distribution as if it were generated
from labeled feature frequencies.
\end{theorem}
%%%

\begin{proof}
First, consider the claim that every labeled feature frequency model has a labeled EFPF.
This claim is intuitively clear since the independent Bernoulli draws 
at each atom of the (potentially random) measure
$\bproc = \sum_{k=1}^{\infty} V_{k} \delta_{\randlabfa_{k}}$
result in a probability that depends only on the number of occurrences
of the corresponding feature and not any interactions between features.

To show this direction formally, we consider a fixed, labeled feature allocation
$\labdetfa_{N} = (\blockfa_{N,1},\blockfa_{N,2},\ldots,\blockfa_{N,K})$
with $M_{N,k} := |\blockfa_{N,k}|$
and note that 
\begin{align*}
	\lefteqn{ \mbp(\labrandfa_{N} = \labdetfa_{N}) } \\
		&= \mbe\left[ \mbp(\labrandfa_{N} = \labdetfa_{N} | \bproc ) \right] \\
		&= \mbe\left[ \left( \prod_{k=1}^{K} V_{k}^{M_{N,k}} (1-V_{k})^{N-M_{N,k}} \right) \cdot \left( \prod_{k=K+1}^{\infty} (1-V_{k})^{N} \right) \right].
\end{align*}
It follows that $\mbp(\labrandfa_{N} = \labdetfa_{N})$ has $\oefpf$ form.

Now consider the other direction. We start with a labeled feature allocation $\randfa_{\infty}$. 
In this case, we know that for every labeled feature allocation of $[N]$,
$$
	\labdetfa_{N} = (\blockfa_{N,1},\ldots,\blockfa_{N,K}),
$$
we have that
a function $\oefpf$ exists in the form
\begin{equation}
	\label{eq:order_K_efpf_assume}
	\mbp(\labrandfa_{N} = \labdetfa_{N}) = \oefpf(N,|\blockfa_{N,1}|, \ldots, |\blockfa_{N,K}|),
\end{equation}
with no additional symmetry assumptions for $\oefpf$ and where the block sizes
$M_{N,k} = |\blockfa_{N,k}|$ may be zero.

Let $Z_{n,k}$ be one if $n$ belongs to the $k$th feature (i.e., $n \in \blockfa_{N,k}$) or zero
otherwise. Let $b_{1}, \ldots, b_{k}$ be values in $\{0,1\}$.
Our goal is to show that conditional on some (as yet unknown)
labeled feature frequencies, the probability of feature presence factorizes
as independent Bernoulli draws:
\begin{equation}
	\label{eq:label_factor}
	\mbp(Z_{1,1}=b_{1},\ldots,Z_{1,K}=b_{K} | V_{1}, \ldots, V_{K})
		= \prod_{k=1}^{K} V_{k}^{b_{k}} (1-V_{k})^{1-b_{k}}.
\end{equation}

By the assumption on $\oefpf$, the labeled feature sizes $M_{N,1},\ldots,M_{N,K}$
are sufficient for the distribution of the labeled feature allocation. So we start by considering
\begin{align}
	\nonumber
	\lefteqn{ \mbp(Z_{1,1}=b_{1},\ldots,Z_{1,K}=b_{K} | M_{N,1}, \ldots, M_{N,K}) } \\
		\label{eq:ordered_features_factorize}
		&= \prod_{k=1}^{K} \mbp(Z_{1,k} = b_{k} | Z_{1,1}=b_{1},\ldots,Z_{1,k-1} = b_{k-1}, M_{N,1}, \ldots, M_{N,K})
\end{align}
Let $\xi_{N}$ be the sigma-field of events invariant under permutations of the first
$N$ indices.
Then again since the feature sizes are sufficient for the feature allocation distribution, we have
\begin{align}
	\nonumber
	\lefteqn{ \mbp(Z_{1,k} = b_{k} | Z_{1,1}=b_{1},\ldots,Z_{1,k-1} = b_{k-1}, M_{N,1}, \ldots, M_{N,K}) } \\
		\label{eq:conditional_exch_sigma_field_sizes}
		&= \mbp(Z_{1,k} = b_{k} | Z_{1,1}=b_{1},\ldots,Z_{1,k-1} = b_{k-1}, \xi_{N}) \\
		\nonumber
		&= \mbp(Z_{1,k} = b_{k} | \xi_{N}) \\
		\nonumber
		&= \frac{1}{N} \sum_{n=1}^{N} \mbp(Z_{n,k} = b_{k} | \xi_{N}) \\
		\nonumber
		&= \mbe\left[ \frac{1}{N} \sum_{n=1}^{N} \mbo\{Z_{n,k} = b_{k}\} | \xi_{N} \right] \\
		\nonumber
		&= \frac{1}{N} \sum_{n=1}^{N} \mbo\{Z_{n,k} = b_{k}\}.
\end{align}
The last line follows since the sum is measurable in $\xi_{N}$.
By the strong law of large numbers, the final sum converges almost surely as
$N \rightarrow \infty$ to some 
potentially random value in $[0,1]$; call it $V_{k}$ if $b_{k}=1$. 
By \eq{ordered_features_factorize}, then, we have
\begin{equation}
	\label{eq:label_cond_sizes_conv_freqs}
	\mbp(Z_{1,1}=b_{1},\ldots,Z_{1,K}=b_{K} | M_{N,1}, \ldots, M_{N,K})
		\convas \prod_{k=1}^{K} V_{k}^{b_{k}} (1-V_{k})^{1-b_{k}}
\end{equation}

On the other hand, \eqs{conditional_exch_sigma_field_sizes} and \eqss{ordered_features_factorize}
imply that 
\begin{align*}
	\lefteqn{\mbp(Z_{1,1}=b_{1},\ldots,Z_{1,K}=b_{K} | M_{N,1}, \ldots, M_{N,K})} \\
		&= \mbp(Z_{1,1}=b_{1},\ldots,Z_{1,K}=b_{K} | \xi_{N}).
\end{align*}
We next observe that the righthand side of the above equality is a reverse martingale.
 $(\xi_{N})$
is a reversed filtration since $\xi_{N} \supseteq \xi_{N+1}$ for all $N$. Moreover,
(1) $\mbp(Z_{1,1}=b_{1},\ldots,Z_{1,K}=b_{K} | \xi_{N})$ is measurable with respect to
$\xi_{N}$; (2) the same quantity is integrable; and (3) by the tower law,
$$
	\mbp(Z_{1,1}=b_{1},\ldots,Z_{1,K}=b_{K} | \xi_{N}) | \xi_{N+1}
		= \mbp(Z_{1,1}=b_{1},\ldots,Z_{1,K}=b_{K} | \xi_{N+1}).
$$

Since $\mbp(Z_{1,1}=b_{1},\ldots,Z_{1,K}=b_{K} | \xi_{N})$ is a reverse martingale, we have
that
$$
	\mbp(Z_{1,1}=b_{1},\ldots,Z_{1,K}=b_{K} | \xi_{N})
		\convas \mbp(Z_{1,1}=b_{1},\ldots,Z_{1,K}=b_{K} | \xi_{\infty})
$$
for $\xi_{\infty} = \bigcap_{n=1}^{\infty} \xi_{n}$ by reverse martingale convergence.
Together with \eq{label_cond_sizes_conv_freqs}, this convergence implies that
$$
	\mbp(Z_{1,1}=b_{1},\ldots,Z_{1,K}=b_{K} | \xi_{\infty})
		= \prod_{k=1}^{K} V_{k}^{b_{k}} (1-V_{k})^{1-b_{k}},
$$
and since the $V_{k}$ are measurable with respect to $\xi_{\infty}$, the 
tower law yields \eq{label_factor}, as was to be shown.
\end{proof}

While illustrative, the two previous results do not directly deal with feature allocations
as defined earlier in this paper; namely, they do not show any equivalence between EFPFs
and feature frequency models in the case where the features are unlabeled (which is exactly the case
where EFPFs are defined). We will show
in the unlabeled case that every feature frequency model has an EFPF and that every regular feature allocation
with an EFPF is an feature frequency model. In fact, we can consider a general---i.e., not necessarily regular---feature allocation and characterize the EFPF representation in this case.

%%%
\begin{theorem}
\label{thm:efpf_freq}
Let $\lambda$ be a non-negative random variable (which may have some arbitrary joint law with the feature frequencies in a feature frequency model). We can obtain an exchangeable feature allocation by generating a feature allocation from a feature frequency model and then, for each index $n$, including an independent $\pois(\lambda)$-distributed number of features of the form $\{n\}$ in addition to those features previously generated (which may also include index $n$). A feature allocation of this type has an EFPF. Conversely, every feature allocation with an EFPF has the same distribution as one generated by this construction for some joint distribution of $\lambda$ and the feature frequencies.
\end{theorem}
%%%

\begin{proof}
Suppose a feature allocation $\labdetfa$ is generated as described by the construction in \thm{efpf_freq}
with (potentially random) measure $\bproc = \sum_{k=1}^{\infty} V_{k} \delta_{\randlabfa_{k}}$
giving the frequencies in the feature frequency model component. We wish to show that the feature allocation
has an EFPF.
We will make use of the fact that
an equivalent way to generate the Poisson component of the feature allocation is to draw
$
	\pois\left( N \lambda \right)
$
singletons and then assign each uniformly at random to an index in $[N]$.

Consider $\labdetfa_{N} = (\blockfa_{1},\blockfa_{2},\ldots,\blockfa_{K})$.
Let $S = \{k: |\blockfa_{k}| = 1\}$ represent the feature indices of the singletons of the feature allocation.
These features may have been generated either from the feature frequency model or
from the Poisson component. To find the probability of the feature allocation,
we consider each possible association of singletons to
one of these components. For any such association, let $\tilde{S}$ represent
those singletons assigned to the Poisson component; that is, $\tilde{S} \subseteq S$.
Let $\tilde{K} = K - |\tilde{S}|$
represent the number of remaining features, which we denote by
$$
	(\tilde{\blockfa}_{1}, \ldots, \tilde{\blockfa}_{\tilde{K}}).
$$
Then the probability of this feature allocation satisfies
\begin{align*}
	\lefteqn{ \mbp(\labrandfa_{N} = \labdetfa_{N}) } \\
		&= \mbe\left[ \mbp(\labrandfa_{N} = \labdetfa_{N} | \bproc, \lambda ) \right] \\
		&= \mbe\left[ \sum_{
			\tilde{S}: \tilde{S} \subseteq S
			}
			N^{-\tilde{S}} \pois\left( \tilde{S} | N \lambda \right)
			\sum_{\stackrel{(i_{1},\ldots,i_{\tilde{K}})}{ \textrm{distinct}}}
			\right. \\
		& \left. 			
				\frac{1}{K!} \left(
					V_{i_{1}}^{|\tilde{\blockfa}_{1}|} (1-V_{i_{1}})^{N-|\tilde{\blockfa}_{1}|} \cdots V_{i_{\tilde{K}}}^{|\tilde{\blockfa}_{\tilde{K}}|} (1-V_{i_{\tilde{K}}})^{N-|\tilde{\blockfa}_{\tilde{K}}|} \prod_{\stackrel{l \in \mathbb{N}}{l \notin \{i_{1}, \ldots, i_{\tilde{K}}\}}} (1-V_{l})^{N}
						\right)
			\right].
		%% sum isn't ordered since uniform reordering at end
\end{align*}
The final expression depends only on the number of data points $N$ and feature sizes
and is symmetric in the feature sizes. So it has EFPF form.

In the other direction, we sidestep the issue of feature ordering by
looking at the number of features to which each data index belongs.
The advantage of this approach is that this
number does not depend on the feature order. The following result
is the key to making use of this observation.

%%%
\begin{lemma} \label{lem:seq_bin}
Let $K_{n}$ be a sequence of positive integers. For each $n$, suppose we have (constants)
$$
	1 \ge p_{n,1} \ge p_{n,2} \ge \ldots \ge p_{n,K_{n}} > 0.
$$
And, for completeness, suppose $p_{n,k} = 0$ for $k > K_{n}$.
Let $X_{n,k} \sim \bern(p_{n,k})$, independently across $n$ and $k$ and with $k = 1:K_{n}$.
Define $\#_{n} := \sum_{k=1}^{K_{n}} X_{n,k}$. 
Then the following are equivalent.
\begin{enumerate}
	\item $\#_{n} \stackrel{d}{\rightarrow} \#$ for some finite-valued random variable $\#$ on $\{0,1,2,\ldots\}$.
	\item There exist (constants) $\{p_{k}\}_{k=1}^{\infty}$ and $\lambda$ such that $p_{k} \in [0,1]$ and $\lambda > 0$ and further such that, $\forall k = 1,2,\ldots$,
	\begin{equation}
		\label{eq:pnk_lim}
		p_{n,k} \rightarrow p_{k}, \quad n \rightarrow \infty
	\end{equation}
	and
	\begin{equation}
		\label{eq:pnk_sum_lim}
		\sum_{k = 1}^{K_{n}} p_{n,k} \rightarrow \sum_{k=1}^{\infty} p_{k} + \lambda, \quad n \rightarrow \infty.
	\end{equation}
\end{enumerate}
In this case, we further have
\begin{equation}
	\label{eq:p_mono}
	1 \ge p_{1} \ge p_{2} \ge \cdots,
\end{equation}
and
\begin{equation}
	\label{eq:num_distr}
	\# \stackrel{d}{=} Y + \sum_{k=1}^{\infty} X_{k},
\end{equation}
where $X_{k} \sim \bern(p_{k})$, independently across $k$, and $Y \sim \pois(\lambda)$.
\end{lemma}
%%%

The proof of \lem{seq_bin} appears in \app{seq_bin}; this lemma is essentially a special case
of a more general result in \app{lemmas_seq_bin}.

In this direction of the proof of \thm{efpf_freq}, we want to show that if we assume that the probability of a feature allocation takes EFPF form, then the allocation has the same distribution as if it were generated according to a feature frequency model with a Poisson-distributed number of singleton features for each $n$. To see how \lem{seq_bin} may be useful, we let $\hat{\#}$ be the number of features in which index 1 occurs. Recall that in order to use the EFPF, we apply a uniform random ordering to the features of our feature allocation. Examining $\hat{\#}$ is advantageous since it is invariant to the ordering of the features, and we can thereby avoid complicated considerations that may arise related to the feature ordering and consistency of ordering across feature allocations of increasing index sets.

Indeed, recall that once we have chosen a uniform random ordering for the features, the EFPF assumption tells us that any feature allocation with the requisite feature sizes and number of indices has the same probability. Let $K_{N}$ be the number of features containing indices $[N]$. If $M_{N,k}$ is the size of the $k$th feature (under the uniform random ordering) after $N$ indices, then there are 
$$
	\binom{N}{M_{N,1}} \cdots \binom{N}{M_{N,K_{N}}}
$$
such configurations. $M_{N,1} / N$ have index 1 in the first feature. For each such allocation, there are equally many configurations of the remaining features. So, for each such allocation, $M_{N,2}/N$ have index 1 in the second feature. And so on. That is, we have that, conditionally on the feature sizes, the number of features with index 1 has the same distribution as a sum of Bernoulli random variables:
\begin{equation}
	\label{eq:unif_ord_sum_bern}
	\sum_{k=1}^{K_{N}} \tilde{X}_{N,k}, \quad \tilde{X}_{N,k} \indep \bern(M_{N,k} / N).
\end{equation}
First, we note that the feature sizes are sufficient for the distribution by the EFPF assumption. So we may, in fact, condition on $\xi_{N}$, which we define to be the sigma-field of events invariant under permutations of the indices $n=1,\ldots,N$. That is, $\hat{\#} | \xi_{N}$ has the same distribution as the sum in \eq{unif_ord_sum_bern}.

Second, we note that the sum in \eq{unif_ord_sum_bern} has no dependence on the ordering of the features. In particular, then, let $1 \ge p_{N,1} \ge p_{N,2} \ge \cdots \ge p_{N,K_{N}}$
be the sizes of the features divided by $N$ and ordered so as to be monotonically decreasing. Again, note that we are only considering those features including some data index in $[N]$. It follows that
\begin{equation}
	\label{eq:efpf_d}
	\hat{\#} | \xi_{N} \stackrel{d}{=} \sum_{k=1}^{K_{N}} \tilde{X}_{N,k}, \quad \tilde{X}_{N,k} \indep \bern(p_{N,k}).
\end{equation}
So we see that we have circumvented ordering concerns and can simply use a size ordering in what follows.

At this point, it seems natural to apply \lem{seq_bin} to $\hat{\#} | \xi_{N}$. To do so, 
we need to show that $\hat{\#} | \xi_{N}$ converges in distribution to some random variable with
non-negative integer values as $N \rightarrow \infty$. To that end, we note that $(\xi_{N})$
is a reversed filtration: $\xi_{N} \supseteq \xi_{N+1}$ for all $N$. And further
$\mbp(\hat{\#} = j | \xi_{N})$ is a reversed martingale since (1) $\mbp(\hat{\#} = j | \xi_{N})$ is measurable
with respect to $\xi_{N}$; (2) $\mbp(\hat{\#} = j | \xi_{N})$ is integrable; and (3) by the tower law, $\mbp(\hat{\#} = j | \xi_{N}) | \xi_{N+1} = \mbp(\hat{\#} = j | \xi_{N+1})$.
It follows that
$$
	\mbp(\hat{\#} = j | \xi_{N}) \convas \mbp(\hat{\#} = j | \xi_{\infty})
$$
and hence
$$
	\hat{\#} | \xi_{N} \convd \hat{\#} | \xi_{\infty} \quad \mathrm{a.s.}\
$$
for $\xi_{\infty} = \bigcap_{n=1}^{\infty} \xi_{n}$ by reverse martingale convergence.

So we may apply \lem{seq_bin} conditional on $\xi_{\infty}$.
By the lemma, we have that, conditional on $\xi_{\infty}$, 
\begin{align*}
	\hat{\#}
		&\stackrel{d}{=} Y + \sum_{k=1}^{\infty} X_{k} \\
	Y &\sim \pois(\lambda) \\
	X_{k} &\indep \bern(p_{k})
\end{align*}
for some $\lambda \ge 0$ and some $1 \ge p_{1} \ge p_{2} \ge \cdots$.
The conditioning on $\xi_{\infty}$ means that, in general,
$\lambda$ and the frequencies
$1 \ge p_{1} \ge p_{2} \ge \cdots$ may be positive random variables,
as was to be shown.
\end{proof}

%%%%%%%%%%%%%%%%%%%%%%%%%%%%
\section{Conclusion} \label{sec:conclusion}
%%%%%%%%%%%%%%%%%%%%%%%%%%%%

It has been known for some time that the class of exchangeable partitions
is the same as the class of partitions generated by the Kingman paintbox,
which is in turn the same as the class of partitions with exchangeable partition
probability functions (EPPFs). In this paper, we have developed an analogous
set of concepts for the feature allocation problem.  We defined a feature allocation
as an extension of partitions in which indices may belong to multiple groups, now
called features.  We have developed analogues of the EPPF and the Kingman paintbox,
which we refer to as the exchangeable feature partition function (EFPF) and the
feature paintbox, respectively.  The feature paintbox allows us to construct
a feature allocation via iid draws from an underlying collection of sets in
the unit interval.  In the special cases of partitions and feature frequency models the construction
of these sets is particularly straightforward.

The Venn diagram presented earlier in \fig{venn_summary} summarizes our
results and also suggests a number of open areas for further investigation.
In particular it would be useful to develop a fuller understanding of the 
regularity condition on feature allocations that allows the connection to
the feature paintbox.  It would also be of interest to carry the program
further by exploring generalizations of the partition and feature allocation 
framework to other combinatorial representations, such as the setting in
which we allow multiplicity within, as well as across, features
\citep{broderick:2011:combinatorial,zhou:2012:beta}.

%%%
\section{Acknowledgments}
T.~Broderick was funded by a National Science Foundation 
Graduate Research Fellowship. 
The work of J.~Pitman was supported in part by National Science Foundation Award 0806118.
Our work has also been supported
by the Office of Naval Research under contract/grant number 
N00014-11-1-0688.
%%%

%%%
\bibliography{bnp}
%%%

%%%
\appendix

%%%%%%%%%%%%%%%%%%%%%%%%%%%%
\section{Intermediate lemmas leading to \lem{seq_bin}} \label{app:lemmas_seq_bin}
%%%%%%%%%%%%%%%%%%%%%%%%%%%%

To prove \lem{seq_bin}, we will make use of a few definitions and lemmas.
We start with two definitions. First, suppose we have constants 
$p_{1}, p_{2}, p_{3}, \ldots$ such that
$$
	1 \ge p_{1} \ge p_{2} \ge p_{3} \ge \ldots \ge 0
$$
and a constant $\lambda$ such that $0 \le \lambda < \infty$.
Then we say that the random variable \# has the \emph{extended Poisson-binomial
distribution} with parameters $(\lambda, p_{1}, p_{2}, \ldots)$ if there
exist independent random variables $X_{0}, X_{1}, X_{2}, \ldots$ with
\begin{align*}
	X_{0} &\sim \pois(\lambda) \\
	X_{k} &\sim \bern(p_{k}), \quad k = 1,2,\ldots
\end{align*}
such that
$$
	\# = X_{0} + \sum_{k=1}^{\infty} X_{k}.
$$

Second, we say that $\mu$ is the \emph{spike size-location measure} with
parameters $(\lambda, p_{1}, p_{2}, \ldots)$ if $\mu$ puts mass $\lambda$ at 0 and mass $p_{k}$ at $p_{k}$ for $k = 1,2,\ldots$.
With these definitions in hand, we can state the following lemmas.

%%%
\begin{lemma} \label{lem:poiss_bin}
Let \# have the extended Poisson-binomial distribution with parameters $(\lambda, p_{1}, p_{2}, \ldots)$.

Then 
\begin{enumerate}
	\item \# is a.s.\ finite if and only if $\sum_{k=1}^{\infty} p_{k} < \infty$
	\item If \# is a.s.\ finite, then the parameters $(\lambda, p_{1}, p_{2}, \ldots)$ are uniquely determined by the distribution of \#.
\end{enumerate}
\end{lemma}
%%%

In particular, since the parameters $(\lambda, p_{1}, p_{2}, \ldots)$ uniquely determine the distribution of \#, \lem{poiss_bin} tells us that there is a bijection between the distribution of \# and the parameters $(\lambda, p_{1}, p_{2}, \ldots)$ when $\#$ is a.s.\ finite. See \app{poiss_bin} for the proof of \lem{poiss_bin}.

The next lemma tells us that this correspondence between distributions and parameters is also continuous in a sense.
%%%
\begin{lemma} \label{lem:conts_bin}
	For $n=1,2,\ldots$, let $\#_{n}$ have the extended Poisson-binomial distribution with 
parameters $(\lambda_{n}, p_{n,1}, p_{n,2},\ldots)$. Let $\mu_{n}$ be the spike size-location measure
with parameters $(\lambda_{n}, p_{n,1}, p_{n,2},\ldots)$.

Then the following two statements are equivalent:
\begin{enumerate}
	\item $\#_{n}$ converges in distribution to a finite-valued limit random variable.
	\item $\mu_{n}$ converges weakly to some finite measure on $[0,1]$.
\end{enumerate}
If the convergence holds, the limiting random variable (call it \#) has an extended Poisson-binomial distribution, and the limiting measure (call it $\mu$) is a spike size-location measure. In this case, $\#$ and $\mu$ have the same parameters; call the parameters $(\lambda, p_{1}, p_{2}, \ldots)$.
\end{lemma}
%%%

This lemma is suggested by, and provides an extension to, previous results on triangular arrays of random variables with row sums converging in distribution; cf., \citet{kallenberg:2002:foundations}.
See \app{conts_bin} for the proof of \lem{conts_bin}.

\lem{seq_bin} highlights a special case of \lems{poiss_bin} and \lemss{conts_bin} that we use to prove the equivalence in \thm{efpf_freq}.

%%%%%%%%%%%%%%%%%%%%%%%%%%%%
\section{Proof of \lem{seq_bin}} \label{app:seq_bin}
%%%%%%%%%%%%%%%%%%%%%%%%%%%%

We can rephrase the statement of \lem{seq_bin} in terms of the terminology introduced in \app{lemmas_seq_bin}. In particular, we are given a sequence of random variables $\#_{n}$, where
$\#_{n}$ has an extended Poisson-binomial distribution with parameters $(0, p_{n,1}, p_{n,2}, \ldots, p_{n,K_{n}}, 0, 0, \ldots)$. Then we see that \lem{seq_bin} is essentially a special case of \lem{conts_bin} where $\lambda_{n}$ and all but finitely many of the $p_{n,k}$ are equal to zero. Indeed, the extended Poisson-binomial distribution in exactly this special case is known as the \emph{Poisson-binomial distribution} \citep{wang:1993:number,chen:1997:statistical}.

\paragraph{(1) $\Rightarrow$ (2).} 
We assume that $\#_{n}$ converges in distribution to some finite-valued random variable \#,
and we wish to show that the $p_{n,k}$ converge to some limiting $p_{k}$ as $n \rightarrow \infty$ for each $k$, and likewise that $\sum_{k=1}^{K_{n}} p_{n,k}$ converges to $\sum_{k=1}^{\infty} p_{k} + \lambda$ for some non-negative constant $\lambda$. The $p_{n,k}$ are just the ordered atom sizes of the spike size-location measures $\mu_{n}$ in \lem{conts_bin}. By \lem{conts_bin}, the $\mu_{n}$ converge weakly to some spike size-location measure $\mu$. This convergence yields both the desired convergence of the atom sizes (\eq{pnk_lim}, repeated here)
$$
	p_{n,k} \rightarrow p_{k}, \quad n \rightarrow \infty
$$
and the desired convergence of the total mass of $\mu_{n}$ (\eq{pnk_sum_lim}, repeated here)
$$
	\sum_{k = 1}^{K_{n}} p_{n,k} \rightarrow \sum_{k=1}^{\infty} p_{k} + \lambda, \quad n \rightarrow \infty.
$$

\paragraph{(2) $\Rightarrow$ (1).} Now we assume that the $p_{n,k}$ converge to some limiting $p_{k}$ as $n \rightarrow \infty$ for each $k$, and likewise that $\sum_{k=1}^{K_{n}} p_{n,k}$ converges to $\sum_{k=1}^{\infty} p_{k} + \lambda$ for some appropriate positive constants $\{p_{k}\}, \lambda$. We wish to show that $\#_{n}$ converges in distribution to some finite-valued random variable \#.

The assumed convergences guarantee the weak convergence of the spike size-location measures $\mu_{n}$ to some finite measure on $[0,1]$. \lem{conts_bin} then guarantees that $\#_{n}$ converges in distribution to some finite-valued random variable \#.

\paragraph{Assume (1) and (2).} 
We wish to show that $1 \ge p_{1} \ge p_{2} \ge \ldots$ (\eq{p_mono}), but this result follows
from the monotonicity of the $p_{n,k}$.

\eq{num_distr} in the original lemma statement can be rephrased as wanting to show that \# has the extended Poisson-binomial distribution with parameters $(\lambda, p_{1}, p_{2}, \ldots)$. This follows directly from the final statement in \lem{conts_bin} and our identification of the limiting spike size-location measure $\mu$ as having parameters $(\lambda, p_{1}, p_{2}, \ldots)$ in a previous part of this proof (``(1) $\Rightarrow$ (2)'').
\qed

%%%%%%%%%%%%%%%%%%%%%%%%%%%%
\section{Proof of \lem{poiss_bin}} \label{app:poiss_bin}
%%%%%%%%%%%%%%%%%%%%%%%%%%%%

Throughout we assume that \# has the extended Poisson-binomial distribution with parameters $(\lambda, p_{1}, p_{2}, \ldots)$.

\paragraph{(1).} We want to show that \# is a.s.\ finite if and only if $\sum_{k=1}^{\infty} p_{k} < \infty$. 
Since \# is extended Poisson-binomially distributed, we can write $\# = X_{0} + \sum_{k=1}^{\infty} X_{k}$ for independent $X_{0} \sim \pois(\lambda)$ and $X_{k} \sim \bern(p_{k})$ for $k=1,2,\ldots$.
First suppose $\sum_{k=1}^{\infty} p_{k} < \infty$. Then $\sum_{k=1}^{\infty} X_{k}$ is a.s.\ finite by the Borel-Cantelli lemma. Second, suppose $\sum_{k=1}^{\infty} p_{k} = \infty$. Then $\sum_{k=1}^{\infty} X_{k}$ is a.s.\ infinite by the second Borel-Cantelli lemma.
Since $X_{0}$ is a.s.\ finite by construction, the result follows.

\paragraph{(2).} We want to show that if \# is a.s.\ finite, then the parameters $(\lambda, p_{1}, p_{2}, \ldots)$ are uniquely determined by the distribution of \#. To that end, let $\mu$ be the spike size-location measure with
parameters $(\lambda, p_{1}, p_{2}, \ldots)$ . Note that $\mu$ need not be a probability measure but is finite by the assumption that \# is a.s.\ finite together with part (1) of the lemma.

To better understand the distribution of \#, we write the probability generating function of \#. For $s$ with $0 \le s \le 1$, we have
\begin{align*}
	\mbe s^{\#} &= e^{-\lambda (1-s)} \prod_{k=1}^{\infty} \left[ 1 - (1-s) p_{k} \right],
\end{align*}
which implies that for $s$ with $0 < s \le 1$ we have
\begin{align}
	- \log \mbe s^{\#}
		\label{eq:neg_log_pgf}
		&= \lambda (1-s) - \sum_{k=1}^{\infty} \log \left[ 1 - (1-s) p_{k} \right] \\
		\nonumber
		&= \lambda (1-s) + \sum_{k=1}^{\infty} \sum_{j=1}^{\infty} \frac{1}{j} (1-s)^{j} p_{k}^{j} \\
		\nonumber
		& \textrm{from the Taylor series expansion of the logarithm} \\
		\nonumber
		&= \lambda (1-s) + \sum_{j=1}^{\infty} \frac{1}{j} (1-s)^{j} \sum_{k=1}^{\infty} p_{k}^{j} \\
		\nonumber
		& \textrm{interchanging the order of summation since the summands are non-negative} \\
		\label{eq:sum_int_from_pgf}
		&= (1-s) \mu\{0\} + \sum_{j=1}^{\infty} \frac{1}{j} (1-s)^{j} \int_{(0,1]} x^{j-1} \mu(dx) \\
		\label{eq:moments_from_pgf}
		&= \sum_{j=1}^{\infty} \frac{1}{j} (1-s)^{j} m_{j-1},
\end{align}
where
$$
	m_{j} := \int_{[0,1]} x^{j} \mu(dx)
$$
is the $j$th moment of the measure $\mu$.

Now the distribution of \# uniquely determines the probability generating function of \#, which by \eq{moments_from_pgf} uniquely determines the sequence of moments of the measure $\mu$. In turn,
$\mu$ is a bounded measure on $[0,1]$ and hence uniquely determined by its moments. And the parameters $(\lambda, p_{1}, p_{2}, \ldots)$ are uniquely determined by $\mu$.
\qed

%%%%%%%%%%%%%%%%%%%%%%%%%%%%
\section{Proof of \lem{conts_bin}} \label{app:conts_bin}
%%%%%%%%%%%%%%%%%%%%%%%%%%%%

For $n=1,2,\ldots$, we assume $\#_{n}$ has the extended Poisson-binomial distribution with 
parameters $(\lambda_{n}, p_{n,1}, p_{n,2},\ldots)$. We further assume $\mu_{n}$ has the spike size-location measure
with parameters $(\lambda_{n}, p_{n,1}, p_{n,2},\ldots)$.

\paragraph{(2) $\Rightarrow$ (1).} Suppose the $\mu_{n}$ converge weakly to some finite measure $\mu$ on $[0,1]$. We want to show that $\#_{n}$ converges in distribution to a finite-valued limit random variable.

In \app{poiss_bin}, we noted that we can express the probability generating function of an extended Poisson-binomial distribution in terms of a spike size-location measure with the same parameters. In particular, by \eq{sum_int_from_pgf}, we can write the negative log of the probability generating function of $\#_{n}$ as
$$
	-\log \mbe s^{\#_{n}}
		= \int_{[0,1]} f_{s}(x) \; \mu_{n}(dx),
$$
where
\begin{equation}
	\label{eq:f_s_x}
	f_{s}(x)
		:= \sum_{j=1}^{\infty} \frac{1}{j} (1-s)^{j} x^{j-1}
		= \left\{ \begin{array}{ll}
				- x^{-1} \log\left[ 1 - (1-s) x \right]		& x > 0 \\
				1-s								& x = 0
			\end{array} \right. .
\end{equation}
Since $f_{s}(x)$ is bounded in $x$ for each fixed $s$ with $0 < s \le 1$, we have by the 
assumption of weak convergence of $\mu_{n}$ that
\begin{align*}
	\lim_{n \rightarrow \infty} -\log \mbe s^{\#_{n}}
		&= \int_{[0,1]} f_{s}(x) \; \mu(dx),
\end{align*}
Moreover, since $\mu$ is finite by assumption,
we have that the result is finite for each $s$ with $0 < s \le 1$.
It follows that $\#_{n}$ converges in distribution to a finite random variable $\#$, with
probability generating function given by
\begin{equation}
	\label{eq:pgf_num_mu}
	\mbe s^{\#} = \exp\left\{ - \int_{[0,1]} f_{s}(x) \; \mu(dx) \right\}.
\end{equation}

\paragraph{Assume (1).} Now suppose the $\#_{n}$ converge in distribution
to a finite random variable $\#$. The next two parts of the proof will rely on an intermediate step: showing that $\mu_{n}$ has bounded total mass in this case.

To show that $\mu_{n}$ has bounded total mass,
first note that $\mbe \#_{n}$ is
exactly the total mass of $\mu_{n}$:
\begin{align*}
	\mbe \#_{n} &= \lambda_{n} + \sum_{k=1}^{\infty} p_{n,k} =: \Sigma_{n}, \\
	\textrm{and } \var \#_{n} &= \lambda_{n} + \sum_{k=1}^{\infty} p_{n,k} (1-p_{n,k}).
\end{align*}
Noting that $\var \#_{n} \le \Sigma_{n}$ allows us to apply Chebyshev's inequality to find
\begin{align*}
	1/4 &\ge \mbp( |\#_{n} - \mbe \#_{n}| \ge 2 \sqrt{\var \#_{n}} ) \\
	3/4 &\le \mbp( |\#_{n} - \Sigma_{n}| \le 2 \sqrt{\var \#_{n}} ) \\
		&\le \mbp( |\#_{n} - \Sigma_{n}| \le 2 \sqrt{\Sigma_{n}} ) \\
		&\le \mbp( \#_{n} \ge \Sigma_{n} - 2 \sqrt{\Sigma_{n}} ).
\end{align*}

Since $\#_{n}$ converges in distribution by assumption, the sequence $\#_{n}$ is tight. Choose $\epsilon$ 
such that $1/2 > \epsilon > 0$. Then there exists some $N_{\epsilon}$ such that, for all $n \ge 1$, we have
$\mbp(\#_{n} \le N_{\epsilon}) > 1 - \epsilon > 1/2$.
It follows that, for all $n \ge 1$,
$$
	1/4 \le \mbp(N_{\epsilon} \ge \Sigma_{n} - 2 \Sigma_{n}).
$$
Since $\Sigma_{n}$ is non-random, it must be that $\mbp(N_{\epsilon} \ge \Sigma_{n} - 2\sqrt{\Sigma_{n}}) = 1$. That is, the total mass of $\mu_{n}$ is bounded.

\paragraph{Assume (1) and (2).} Suppose $\#_{n}$ converges in distribution to some finite-valued limit random variable \# and that $\mu_{n}$ converges weakly to some finite measure $\mu$.
We want to show that $\#$ has an extended Poisson-binomial distribution, that $\mu$ is a spike size-location measure, and that $\#$ and $\mu$ have the same parameters.

We start by showing that $\mu$ is discrete.
Choose any $\epsilon > 0$. Since the mass of $\mu_{n}$ is bounded across $n$ by the previous part of the proof (``Assume (1)''), the number of 
atoms of $\mu_{n}$ greater than $\epsilon$ is bounded across $n$. It follows that the number of
atoms of $\mu$ has the same bound. So $\mu$ is discrete. Since $\mu_{n}$ converges weakly to
$\mu$, we see that $\mu$ must have atoms with sizes and locations $p_{1}, p_{2}, \ldots$ such that
$$
	1 \ge p_{1} \ge p_{2} \ge \ldots
$$
as well as a potential atom, with size we denote by $\lambda$, at zero. 
That is, $\mu$ is a spike size-location measure with parameters $(\lambda, p_{1}, p_{2}, \ldots)$.

In a previous part of the proof (``(2) $\Rightarrow$ (1)''), we 
expressed the probability generating function
of $\#$ as a function of $\mu$ (\eq{pgf_num_mu}).
With this relation in hand, we can reverse the series of equations
presented in \app{poiss_bin} and
ending in \eq{sum_int_from_pgf} to find the form of the probability
generating function for \#
(\eq{neg_log_pgf}). In particular, \eq{neg_log_pgf} tells
us that \# is an extended Poisson-binomial random variable
with parameters
$(\lambda, p_{1}, p_{2}, \ldots)$. In particular,
we emphasize that \# has the same parameters as $\mu$, which 
we have already shown above is a spike size-location measure.

\paragraph{(1) $\Rightarrow$ (2)} Now step back and assume that $\#_{n}$ converges in distribution to a finite-valued limit random variable; call it \#. We wish to show that $\mu_{n}$ converges weakly to some finite measure on $[0,1]$.

%We wish to show that $\mu_{n}$ is a tight sequence in order to apply
%Prokhorov's Theorem 
By a previous part of this proof (``Assume (1)''),
the mass of $\mu_{n}$ is bounded across $n$. Moreover, by construction,
all of the mass for each $\mu_{n}$ is concentrated on $[0,1]$. So it must be that
the sequence $\mu_{n}$ is tight. It follows that if every weakly convergent subsequence $\mu_{n_{j}}$
has the same limit $\mu$, then $\mu_{n}$ converges weakly to $\mu$. 

Consider a subsequence $(n_{j})_{j}$ of $\mathbb{N}$. We know
$\#_{n_{j}}$ converges in distribution to $\#$ by the assumption that $\#_{n}$ converges in distribution to \#.  The previous part of this proof (``Assume (1) and (2)'')
gives that the form of the limit of $\mu_{n_{j}}$ is determined by $\#$; namely, the limit
is a spike size-location measure with parameters shared by \#. In particular, then, the limit $\mu$
must be the same for every subsequence, and the desired result is shown.
\qed

\end{document}

%% file: venn_summary.pstex_t
\begin{picture}(0,0)%
\includegraphics{venn_summary.pstex}%
\end{picture}%
\setlength{\unitlength}{4144sp}%
\begingroup\makeatletter\ifx\SetFigFont\undefined%
\gdef\SetFigFont#1#2#3#4#5{%
  \reset@font\fontsize{#1}{#2pt}%
  \fontfamily{#3}\fontseries{#4}\fontshape{#5}%
  \selectfont}%
\fi\endgroup%
\begin{picture}(9883,4590)(2689,-5116)
\put(10081,-3976){\makebox(0,0)[lb]{\smash{{\SetFigFont{20}{24.0}{\rmdefault}{\mddefault}{\updefault}{\color[rgb]{0,0,0}Regular FAs}%
}}}}
\put(10081,-3121){\makebox(0,0)[lb]{\smash{{\SetFigFont{20}{24.0}{\rmdefault}{\mddefault}{\updefault}{\color[rgb]{0,0,0}$=$ Frequency models}%
}}}}
\put(10081,-2806){\makebox(0,0)[lb]{\smash{{\SetFigFont{20}{24.0}{\rmdefault}{\mddefault}{\updefault}{\color[rgb]{0,0,0}FAs with EFPFs}%
}}}}
\put(10081,-4336){\makebox(0,0)[lb]{\smash{{\SetFigFont{20}{24.0}{\rmdefault}{\mddefault}{\updefault}{\color[rgb]{0,0,0}$=$ Feature paintbox models}%
}}}}
\put(10081,-1501){\makebox(0,0)[lb]{\smash{{\SetFigFont{20}{24.0}{\rmdefault}{\mddefault}{\updefault}{\color[rgb]{0,0,0}Exchangeable RPs}%
}}}}
\put(10081,-2131){\makebox(0,0)[lb]{\smash{{\SetFigFont{20}{24.0}{\rmdefault}{\mddefault}{\updefault}{\color[rgb]{0,0,0}$=$ Kingman paintbox models}%
}}}}
\put(10081,-1816){\makebox(0,0)[lb]{\smash{{\SetFigFont{20}{24.0}{\rmdefault}{\mddefault}{\updefault}{\color[rgb]{0,0,0}$=$ RPs with EPPFs}%
}}}}
\put(10081,-3436){\makebox(0,0)[lb]{\smash{{\SetFigFont{20}{24.0}{\rmdefault}{\mddefault}{\updefault}{\color[rgb]{0,0,0}\;\;plus singletons}%
}}}}
\put(9496,-781){\makebox(0,0)[lb]{\smash{{\SetFigFont{20}{24.0}{\rmdefault}{\mddefault}{\updefault}{\color[rgb]{0,0,0}Exchangeable FAs}%
}}}}
\put(6976,-5011){\makebox(0,0)[lb]{\smash{{\SetFigFont{20}{24.0}{\rmdefault}{\mddefault}{\updefault}{\color[rgb]{0,0,0}IBP}%
}}}}
\put(8146,-5011){\makebox(0,0)[lb]{\smash{{\SetFigFont{20}{24.0}{\rmdefault}{\mddefault}{\updefault}{\color[rgb]{0,0,0}Two-feature example}%
}}}}
\put(5041,-5011){\makebox(0,0)[lb]{\smash{{\SetFigFont{20}{24.0}{\rmdefault}{\mddefault}{\updefault}{\color[rgb]{0,0,0}CRP}%
}}}}
\end{picture}%

%% file: mx_feat2.pstex_t
\begin{picture}(0,0)%
\includegraphics{mx_feat2.pstex}%
\end{picture}%
\setlength{\unitlength}{4144sp}%
\begingroup\makeatletter\ifx\SetFigFont\undefined%
\gdef\SetFigFont#1#2#3#4#5{%
  \reset@font\fontsize{#1}{#2pt}%
  \fontfamily{#3}\fontseries{#4}\fontshape{#5}%
  \selectfont}%
\fi\endgroup%
\begin{picture}(1062,1605)(2101,-1873)
\put(2431,-1276){\makebox(0,0)[lb]{\smash{{\SetFigFont{20}{24.0}{\rmdefault}{\mddefault}{\updefault}{\color[rgb]{0,0,0}1}%
}}}}
\put(2881,-871){\makebox(0,0)[lb]{\smash{{\SetFigFont{20}{24.0}{\rmdefault}{\mddefault}{\updefault}{\color[rgb]{0,0,0}1}%
}}}}
\put(2431,-1726){\makebox(0,0)[lb]{\smash{{\SetFigFont{20}{24.0}{\rmdefault}{\mddefault}{\updefault}{\color[rgb]{0,0,0}2}%
}}}}
\put(2881,-511){\makebox(0,0)[lb]{\smash{{\SetFigFont{20}{24.0}{\rmdefault}{\mddefault}{\updefault}{\color[rgb]{0,0,0}$k$}%
}}}}
\put(2116,-1501){\makebox(0,0)[lb]{\smash{{\SetFigFont{20}{24.0}{\rmdefault}{\mddefault}{\updefault}{\color[rgb]{0,0,0}$n$}%
}}}}
\end{picture}%

%% file: mx_feat3.pstex_t
\begin{picture}(0,0)%
\includegraphics{mx_feat3.pstex}%
\end{picture}%
\setlength{\unitlength}{4144sp}%
\begingroup\makeatletter\ifx\SetFigFont\undefined%
\gdef\SetFigFont#1#2#3#4#5{%
  \reset@font\fontsize{#1}{#2pt}%
  \fontfamily{#3}\fontseries{#4}\fontshape{#5}%
  \selectfont}%
\fi\endgroup%
\begin{picture}(2097,1692)(2416,-2323)
\put(2431,-2176){\makebox(0,0)[lb]{\smash{{\SetFigFont{20}{24.0}{\rmdefault}{\mddefault}{\updefault}{\color[rgb]{0,0,0}3}%
}}}}
\put(2881,-871){\makebox(0,0)[lb]{\smash{{\SetFigFont{20}{24.0}{\rmdefault}{\mddefault}{\updefault}{\color[rgb]{0,0,0}1}%
}}}}
\put(3331,-871){\makebox(0,0)[lb]{\smash{{\SetFigFont{20}{24.0}{\rmdefault}{\mddefault}{\updefault}{\color[rgb]{0,0,0}2}%
}}}}
\put(3781,-871){\makebox(0,0)[lb]{\smash{{\SetFigFont{20}{24.0}{\rmdefault}{\mddefault}{\updefault}{\color[rgb]{0,0,0}3}%
}}}}
\put(4231,-871){\makebox(0,0)[lb]{\smash{{\SetFigFont{20}{24.0}{\rmdefault}{\mddefault}{\updefault}{\color[rgb]{0,0,0}4}%
}}}}
\put(2431,-1276){\makebox(0,0)[lb]{\smash{{\SetFigFont{20}{24.0}{\rmdefault}{\mddefault}{\updefault}{\color[rgb]{0,0,0}1}%
}}}}
\put(2431,-1726){\makebox(0,0)[lb]{\smash{{\SetFigFont{20}{24.0}{\rmdefault}{\mddefault}{\updefault}{\color[rgb]{0,0,0}2}%
}}}}
\end{picture}%

%% file: mx_feat4.pstex_t
\begin{picture}(0,0)%
\includegraphics{mx_feat4.pstex}%
\end{picture}%
\setlength{\unitlength}{4144sp}%
\begingroup\makeatletter\ifx\SetFigFont\undefined%
\gdef\SetFigFont#1#2#3#4#5{%
  \reset@font\fontsize{#1}{#2pt}%
  \fontfamily{#3}\fontseries{#4}\fontshape{#5}%
  \selectfont}%
\fi\endgroup%
\begin{picture}(2547,2142)(2416,-2773)
\put(4681,-871){\makebox(0,0)[lb]{\smash{{\SetFigFont{20}{24.0}{\rmdefault}{\mddefault}{\updefault}{\color[rgb]{0,0,0}5}%
}}}}
\put(2881,-871){\makebox(0,0)[lb]{\smash{{\SetFigFont{20}{24.0}{\rmdefault}{\mddefault}{\updefault}{\color[rgb]{0,0,0}1}%
}}}}
\put(3331,-871){\makebox(0,0)[lb]{\smash{{\SetFigFont{20}{24.0}{\rmdefault}{\mddefault}{\updefault}{\color[rgb]{0,0,0}2}%
}}}}
\put(3781,-871){\makebox(0,0)[lb]{\smash{{\SetFigFont{20}{24.0}{\rmdefault}{\mddefault}{\updefault}{\color[rgb]{0,0,0}3}%
}}}}
\put(4231,-871){\makebox(0,0)[lb]{\smash{{\SetFigFont{20}{24.0}{\rmdefault}{\mddefault}{\updefault}{\color[rgb]{0,0,0}4}%
}}}}
\put(2431,-1276){\makebox(0,0)[lb]{\smash{{\SetFigFont{20}{24.0}{\rmdefault}{\mddefault}{\updefault}{\color[rgb]{0,0,0}1}%
}}}}
\put(2431,-1726){\makebox(0,0)[lb]{\smash{{\SetFigFont{20}{24.0}{\rmdefault}{\mddefault}{\updefault}{\color[rgb]{0,0,0}2}%
}}}}
\put(2431,-2176){\makebox(0,0)[lb]{\smash{{\SetFigFont{20}{24.0}{\rmdefault}{\mddefault}{\updefault}{\color[rgb]{0,0,0}3}%
}}}}
\put(2431,-2626){\makebox(0,0)[lb]{\smash{{\SetFigFont{20}{24.0}{\rmdefault}{\mddefault}{\updefault}{\color[rgb]{0,0,0}4}%
}}}}
\end{picture}%

%% file: mx_feat5.pstex_t
\begin{picture}(0,0)%
\includegraphics{mx_feat5.pstex}%
\end{picture}%
\setlength{\unitlength}{4144sp}%
\begingroup\makeatletter\ifx\SetFigFont\undefined%
\gdef\SetFigFont#1#2#3#4#5{%
  \reset@font\fontsize{#1}{#2pt}%
  \fontfamily{#3}\fontseries{#4}\fontshape{#5}%
  \selectfont}%
\fi\endgroup%
\begin{picture}(2547,2592)(2416,-3223)
\put(4681,-871){\makebox(0,0)[lb]{\smash{{\SetFigFont{20}{24.0}{\rmdefault}{\mddefault}{\updefault}{\color[rgb]{0,0,0}5}%
}}}}
\put(2881,-871){\makebox(0,0)[lb]{\smash{{\SetFigFont{20}{24.0}{\rmdefault}{\mddefault}{\updefault}{\color[rgb]{0,0,0}1}%
}}}}
\put(3331,-871){\makebox(0,0)[lb]{\smash{{\SetFigFont{20}{24.0}{\rmdefault}{\mddefault}{\updefault}{\color[rgb]{0,0,0}2}%
}}}}
\put(3781,-871){\makebox(0,0)[lb]{\smash{{\SetFigFont{20}{24.0}{\rmdefault}{\mddefault}{\updefault}{\color[rgb]{0,0,0}3}%
}}}}
\put(4231,-871){\makebox(0,0)[lb]{\smash{{\SetFigFont{20}{24.0}{\rmdefault}{\mddefault}{\updefault}{\color[rgb]{0,0,0}4}%
}}}}
\put(2431,-1276){\makebox(0,0)[lb]{\smash{{\SetFigFont{20}{24.0}{\rmdefault}{\mddefault}{\updefault}{\color[rgb]{0,0,0}1}%
}}}}
\put(2431,-1726){\makebox(0,0)[lb]{\smash{{\SetFigFont{20}{24.0}{\rmdefault}{\mddefault}{\updefault}{\color[rgb]{0,0,0}2}%
}}}}
\put(2431,-2176){\makebox(0,0)[lb]{\smash{{\SetFigFont{20}{24.0}{\rmdefault}{\mddefault}{\updefault}{\color[rgb]{0,0,0}3}%
}}}}
\put(2431,-2626){\makebox(0,0)[lb]{\smash{{\SetFigFont{20}{24.0}{\rmdefault}{\mddefault}{\updefault}{\color[rgb]{0,0,0}4}%
}}}}
\put(2431,-3076){\makebox(0,0)[lb]{\smash{{\SetFigFont{20}{24.0}{\rmdefault}{\mddefault}{\updefault}{\color[rgb]{0,0,0}5}%
}}}}
\end{picture}%

%% file: mx_feat6.pstex_t
\begin{picture}(0,0)%
\includegraphics{mx_feat6.pstex}%
\end{picture}%
\setlength{\unitlength}{4144sp}%
\begingroup\makeatletter\ifx\SetFigFont\undefined%
\gdef\SetFigFont#1#2#3#4#5{%
  \reset@font\fontsize{#1}{#2pt}%
  \fontfamily{#3}\fontseries{#4}\fontshape{#5}%
  \selectfont}%
\fi\endgroup%
\begin{picture}(2547,3042)(2416,-3673)
\put(4681,-871){\makebox(0,0)[lb]{\smash{{\SetFigFont{20}{24.0}{\rmdefault}{\mddefault}{\updefault}{\color[rgb]{0,0,0}5}%
}}}}
\put(2881,-871){\makebox(0,0)[lb]{\smash{{\SetFigFont{20}{24.0}{\rmdefault}{\mddefault}{\updefault}{\color[rgb]{0,0,0}1}%
}}}}
\put(3331,-871){\makebox(0,0)[lb]{\smash{{\SetFigFont{20}{24.0}{\rmdefault}{\mddefault}{\updefault}{\color[rgb]{0,0,0}2}%
}}}}
\put(3781,-871){\makebox(0,0)[lb]{\smash{{\SetFigFont{20}{24.0}{\rmdefault}{\mddefault}{\updefault}{\color[rgb]{0,0,0}3}%
}}}}
\put(4231,-871){\makebox(0,0)[lb]{\smash{{\SetFigFont{20}{24.0}{\rmdefault}{\mddefault}{\updefault}{\color[rgb]{0,0,0}4}%
}}}}
\put(2431,-1276){\makebox(0,0)[lb]{\smash{{\SetFigFont{20}{24.0}{\rmdefault}{\mddefault}{\updefault}{\color[rgb]{0,0,0}1}%
}}}}
\put(2431,-1726){\makebox(0,0)[lb]{\smash{{\SetFigFont{20}{24.0}{\rmdefault}{\mddefault}{\updefault}{\color[rgb]{0,0,0}2}%
}}}}
\put(2431,-2176){\makebox(0,0)[lb]{\smash{{\SetFigFont{20}{24.0}{\rmdefault}{\mddefault}{\updefault}{\color[rgb]{0,0,0}3}%
}}}}
\put(2431,-2626){\makebox(0,0)[lb]{\smash{{\SetFigFont{20}{24.0}{\rmdefault}{\mddefault}{\updefault}{\color[rgb]{0,0,0}4}%
}}}}
\put(2431,-3076){\makebox(0,0)[lb]{\smash{{\SetFigFont{20}{24.0}{\rmdefault}{\mddefault}{\updefault}{\color[rgb]{0,0,0}5}%
}}}}
\put(2431,-3526){\makebox(0,0)[lb]{\smash{{\SetFigFont{20}{24.0}{\rmdefault}{\mddefault}{\updefault}{\color[rgb]{0,0,0}6}%
}}}}
\end{picture}%

%% file: rand_lab.pstex_t
\begin{picture}(0,0)%
\includegraphics{rand_lab.pstex}%
\end{picture}%
\setlength{\unitlength}{4144sp}%
\begingroup\makeatletter\ifx\SetFigFont\undefined%
\gdef\SetFigFont#1#2#3#4#5{%
  \reset@font\fontsize{#1}{#2pt}%
  \fontfamily{#3}\fontseries{#4}\fontshape{#5}%
  \selectfont}%
\fi\endgroup%
\begin{picture}(4797,3969)(2416,-4978)
\put(3781,-1276){\makebox(0,0)[lb]{\smash{{\SetFigFont{20}{24.0}{\rmdefault}{\mddefault}{\updefault}{\color[rgb]{0,0,0}$\phi_{5}$}%
}}}}
\put(4141,-1276){\makebox(0,0)[lb]{\smash{{\SetFigFont{20}{24.0}{\rmdefault}{\mddefault}{\updefault}{\color[rgb]{0,0,0}$\phi_{2}$}%
}}}}
\put(5716,-1276){\makebox(0,0)[lb]{\smash{{\SetFigFont{20}{24.0}{\rmdefault}{\mddefault}{\updefault}{\color[rgb]{0,0,0}$\phi_{1}$}%
}}}}
\put(6661,-1276){\makebox(0,0)[lb]{\smash{{\SetFigFont{20}{24.0}{\rmdefault}{\mddefault}{\updefault}{\color[rgb]{0,0,0}$\phi_{3}$}%
}}}}
\put(2791,-1276){\makebox(0,0)[lb]{\smash{{\SetFigFont{20}{24.0}{\rmdefault}{\mddefault}{\updefault}{\color[rgb]{0,0,0}$\phi_{4}$}%
}}}}
\end{picture}%

%% file: rmx_feat3.pstex_t
\begin{picture}(0,0)%
\includegraphics{rmx_feat3.pstex}%
\end{picture}%
\setlength{\unitlength}{4144sp}%
\begingroup\makeatletter\ifx\SetFigFont\undefined%
\gdef\SetFigFont#1#2#3#4#5{%
  \reset@font\fontsize{#1}{#2pt}%
  \fontfamily{#3}\fontseries{#4}\fontshape{#5}%
  \selectfont}%
\fi\endgroup%
\begin{picture}(2097,1692)(2416,-2323)
\put(2881,-871){\makebox(0,0)[lb]{\smash{{\SetFigFont{20}{24.0}{\rmdefault}{\mddefault}{\updefault}{\color[rgb]{0,0,0}1}%
}}}}
\put(3331,-871){\makebox(0,0)[lb]{\smash{{\SetFigFont{20}{24.0}{\rmdefault}{\mddefault}{\updefault}{\color[rgb]{0,0,0}2}%
}}}}
\put(3781,-871){\makebox(0,0)[lb]{\smash{{\SetFigFont{20}{24.0}{\rmdefault}{\mddefault}{\updefault}{\color[rgb]{0,0,0}3}%
}}}}
\put(4231,-871){\makebox(0,0)[lb]{\smash{{\SetFigFont{20}{24.0}{\rmdefault}{\mddefault}{\updefault}{\color[rgb]{0,0,0}4}%
}}}}
\put(2431,-1276){\makebox(0,0)[lb]{\smash{{\SetFigFont{20}{24.0}{\rmdefault}{\mddefault}{\updefault}{\color[rgb]{0,0,0}1}%
}}}}
\put(2431,-1726){\makebox(0,0)[lb]{\smash{{\SetFigFont{20}{24.0}{\rmdefault}{\mddefault}{\updefault}{\color[rgb]{0,0,0}2}%
}}}}
\put(2431,-2176){\makebox(0,0)[lb]{\smash{{\SetFigFont{20}{24.0}{\rmdefault}{\mddefault}{\updefault}{\color[rgb]{0,0,0}3}%
}}}}
\end{picture}%

%% file: rmx_feat4.pstex_t
\begin{picture}(0,0)%
\includegraphics{rmx_feat4.pstex}%
\end{picture}%
\setlength{\unitlength}{4144sp}%
\begingroup\makeatletter\ifx\SetFigFont\undefined%
\gdef\SetFigFont#1#2#3#4#5{%
  \reset@font\fontsize{#1}{#2pt}%
  \fontfamily{#3}\fontseries{#4}\fontshape{#5}%
  \selectfont}%
\fi\endgroup%
\begin{picture}(2547,2142)(2416,-2773)
\put(2881,-871){\makebox(0,0)[lb]{\smash{{\SetFigFont{20}{24.0}{\rmdefault}{\mddefault}{\updefault}{\color[rgb]{0,0,0}1}%
}}}}
\put(3331,-871){\makebox(0,0)[lb]{\smash{{\SetFigFont{20}{24.0}{\rmdefault}{\mddefault}{\updefault}{\color[rgb]{0,0,0}2}%
}}}}
\put(3781,-871){\makebox(0,0)[lb]{\smash{{\SetFigFont{20}{24.0}{\rmdefault}{\mddefault}{\updefault}{\color[rgb]{0,0,0}3}%
}}}}
\put(4231,-871){\makebox(0,0)[lb]{\smash{{\SetFigFont{20}{24.0}{\rmdefault}{\mddefault}{\updefault}{\color[rgb]{0,0,0}4}%
}}}}
\put(2431,-1276){\makebox(0,0)[lb]{\smash{{\SetFigFont{20}{24.0}{\rmdefault}{\mddefault}{\updefault}{\color[rgb]{0,0,0}1}%
}}}}
\put(2431,-1726){\makebox(0,0)[lb]{\smash{{\SetFigFont{20}{24.0}{\rmdefault}{\mddefault}{\updefault}{\color[rgb]{0,0,0}2}%
}}}}
\put(2431,-2176){\makebox(0,0)[lb]{\smash{{\SetFigFont{20}{24.0}{\rmdefault}{\mddefault}{\updefault}{\color[rgb]{0,0,0}3}%
}}}}
\put(2431,-2626){\makebox(0,0)[lb]{\smash{{\SetFigFont{20}{24.0}{\rmdefault}{\mddefault}{\updefault}{\color[rgb]{0,0,0}4}%
}}}}
\put(4681,-871){\makebox(0,0)[lb]{\smash{{\SetFigFont{20}{24.0}{\rmdefault}{\mddefault}{\updefault}{\color[rgb]{0,0,0}5}%
}}}}
\end{picture}%

%% file: rmx_feat5.pstex_t
\begin{picture}(0,0)%
\includegraphics{rmx_feat5.pstex}%
\end{picture}%
\setlength{\unitlength}{4144sp}%
\begingroup\makeatletter\ifx\SetFigFont\undefined%
\gdef\SetFigFont#1#2#3#4#5{%
  \reset@font\fontsize{#1}{#2pt}%
  \fontfamily{#3}\fontseries{#4}\fontshape{#5}%
  \selectfont}%
\fi\endgroup%
\begin{picture}(2547,2592)(2416,-3223)
\put(2881,-871){\makebox(0,0)[lb]{\smash{{\SetFigFont{20}{24.0}{\rmdefault}{\mddefault}{\updefault}{\color[rgb]{0,0,0}1}%
}}}}
\put(3331,-871){\makebox(0,0)[lb]{\smash{{\SetFigFont{20}{24.0}{\rmdefault}{\mddefault}{\updefault}{\color[rgb]{0,0,0}2}%
}}}}
\put(3781,-871){\makebox(0,0)[lb]{\smash{{\SetFigFont{20}{24.0}{\rmdefault}{\mddefault}{\updefault}{\color[rgb]{0,0,0}3}%
}}}}
\put(4231,-871){\makebox(0,0)[lb]{\smash{{\SetFigFont{20}{24.0}{\rmdefault}{\mddefault}{\updefault}{\color[rgb]{0,0,0}4}%
}}}}
\put(2431,-1276){\makebox(0,0)[lb]{\smash{{\SetFigFont{20}{24.0}{\rmdefault}{\mddefault}{\updefault}{\color[rgb]{0,0,0}1}%
}}}}
\put(2431,-1726){\makebox(0,0)[lb]{\smash{{\SetFigFont{20}{24.0}{\rmdefault}{\mddefault}{\updefault}{\color[rgb]{0,0,0}2}%
}}}}
\put(2431,-2176){\makebox(0,0)[lb]{\smash{{\SetFigFont{20}{24.0}{\rmdefault}{\mddefault}{\updefault}{\color[rgb]{0,0,0}3}%
}}}}
\put(2431,-2626){\makebox(0,0)[lb]{\smash{{\SetFigFont{20}{24.0}{\rmdefault}{\mddefault}{\updefault}{\color[rgb]{0,0,0}4}%
}}}}
\put(2431,-3076){\makebox(0,0)[lb]{\smash{{\SetFigFont{20}{24.0}{\rmdefault}{\mddefault}{\updefault}{\color[rgb]{0,0,0}5}%
}}}}
\put(4681,-871){\makebox(0,0)[lb]{\smash{{\SetFigFont{20}{24.0}{\rmdefault}{\mddefault}{\updefault}{\color[rgb]{0,0,0}5}%
}}}}
\end{picture}%

%% file: rmx_feat6.pstex_t
\begin{picture}(0,0)%
\includegraphics{rmx_feat6.pstex}%
\end{picture}%
\setlength{\unitlength}{4144sp}%
\begingroup\makeatletter\ifx\SetFigFont\undefined%
\gdef\SetFigFont#1#2#3#4#5{%
  \reset@font\fontsize{#1}{#2pt}%
  \fontfamily{#3}\fontseries{#4}\fontshape{#5}%
  \selectfont}%
\fi\endgroup%
\begin{picture}(2547,3042)(2416,-3673)
\put(2881,-871){\makebox(0,0)[lb]{\smash{{\SetFigFont{20}{24.0}{\rmdefault}{\mddefault}{\updefault}{\color[rgb]{0,0,0}1}%
}}}}
\put(3331,-871){\makebox(0,0)[lb]{\smash{{\SetFigFont{20}{24.0}{\rmdefault}{\mddefault}{\updefault}{\color[rgb]{0,0,0}2}%
}}}}
\put(3781,-871){\makebox(0,0)[lb]{\smash{{\SetFigFont{20}{24.0}{\rmdefault}{\mddefault}{\updefault}{\color[rgb]{0,0,0}3}%
}}}}
\put(4231,-871){\makebox(0,0)[lb]{\smash{{\SetFigFont{20}{24.0}{\rmdefault}{\mddefault}{\updefault}{\color[rgb]{0,0,0}4}%
}}}}
\put(2431,-1276){\makebox(0,0)[lb]{\smash{{\SetFigFont{20}{24.0}{\rmdefault}{\mddefault}{\updefault}{\color[rgb]{0,0,0}1}%
}}}}
\put(2431,-1726){\makebox(0,0)[lb]{\smash{{\SetFigFont{20}{24.0}{\rmdefault}{\mddefault}{\updefault}{\color[rgb]{0,0,0}2}%
}}}}
\put(2431,-2176){\makebox(0,0)[lb]{\smash{{\SetFigFont{20}{24.0}{\rmdefault}{\mddefault}{\updefault}{\color[rgb]{0,0,0}3}%
}}}}
\put(2431,-2626){\makebox(0,0)[lb]{\smash{{\SetFigFont{20}{24.0}{\rmdefault}{\mddefault}{\updefault}{\color[rgb]{0,0,0}4}%
}}}}
\put(2431,-3076){\makebox(0,0)[lb]{\smash{{\SetFigFont{20}{24.0}{\rmdefault}{\mddefault}{\updefault}{\color[rgb]{0,0,0}5}%
}}}}
\put(2431,-3526){\makebox(0,0)[lb]{\smash{{\SetFigFont{20}{24.0}{\rmdefault}{\mddefault}{\updefault}{\color[rgb]{0,0,0}6}%
}}}}
\put(4681,-871){\makebox(0,0)[lb]{\smash{{\SetFigFont{20}{24.0}{\rmdefault}{\mddefault}{\updefault}{\color[rgb]{0,0,0}5}%
}}}}
\end{picture}%

%% file: ibp.pstex_t
\begin{picture}(0,0)%
\includegraphics{ibp.pstex}%
\end{picture}%
\setlength{\unitlength}{4144sp}%
\begingroup\makeatletter\ifx\SetFigFont\undefined%
\gdef\SetFigFont#1#2#3#4#5{%
  \reset@font\fontsize{#1}{#2pt}%
  \fontfamily{#3}\fontseries{#4}\fontshape{#5}%
  \selectfont}%
\fi\endgroup%
\begin{picture}(7095,5206)(-419,-5837)
\put(6661,-1411){\makebox(0,0)[lb]{\smash{{\SetFigFont{20}{24.0}{\rmdefault}{\mddefault}{\updefault}{\color[rgb]{0,0,0}$\cdots$}%
}}}}
\put(3331,-1411){\makebox(0,0)[lb]{\smash{{\SetFigFont{20}{24.0}{\rmdefault}{\mddefault}{\updefault}{\color[rgb]{0,0,0}1}%
}}}}
\put(3826,-1411){\makebox(0,0)[lb]{\smash{{\SetFigFont{20}{24.0}{\rmdefault}{\mddefault}{\updefault}{\color[rgb]{0,0,0}2}%
}}}}
\put(4321,-1411){\makebox(0,0)[lb]{\smash{{\SetFigFont{20}{24.0}{\rmdefault}{\mddefault}{\updefault}{\color[rgb]{0,0,0}3}%
}}}}
\put(4816,-1411){\makebox(0,0)[lb]{\smash{{\SetFigFont{20}{24.0}{\rmdefault}{\mddefault}{\updefault}{\color[rgb]{0,0,0}4}%
}}}}
\put(5311,-1411){\makebox(0,0)[lb]{\smash{{\SetFigFont{20}{24.0}{\rmdefault}{\mddefault}{\updefault}{\color[rgb]{0,0,0}5}%
}}}}
\put(5806,-1411){\makebox(0,0)[lb]{\smash{{\SetFigFont{20}{24.0}{\rmdefault}{\mddefault}{\updefault}{\color[rgb]{0,0,0}6}%
}}}}
\put(6301,-1411){\makebox(0,0)[lb]{\smash{{\SetFigFont{20}{24.0}{\rmdefault}{\mddefault}{\updefault}{\color[rgb]{0,0,0}7}%
}}}}
\put(1711,-1951){\makebox(0,0)[lb]{\smash{{\SetFigFont{20}{24.0}{\rmdefault}{\mddefault}{\updefault}{\color[rgb]{0,0,0}1}%
}}}}
\put(2206,-2266){\makebox(0,0)[lb]{\smash{{\SetFigFont{20}{24.0}{\rmdefault}{\mddefault}{\updefault}{\color[rgb]{0,0,0}2}%
}}}}
\put(1531,-2446){\makebox(0,0)[lb]{\smash{{\SetFigFont{20}{24.0}{\rmdefault}{\mddefault}{\updefault}{\color[rgb]{0,0,0}3}%
}}}}
\put(1981,-2536){\makebox(0,0)[lb]{\smash{{\SetFigFont{20}{24.0}{\rmdefault}{\mddefault}{\updefault}{\color[rgb]{0,0,0}4}%
}}}}
\put(1711,-3436){\makebox(0,0)[lb]{\smash{{\SetFigFont{20}{24.0}{\rmdefault}{\mddefault}{\updefault}{\color[rgb]{0,0,0}2}%
}}}}
\put(2161,-3661){\makebox(0,0)[lb]{\smash{{\SetFigFont{20}{24.0}{\rmdefault}{\mddefault}{\updefault}{\color[rgb]{0,0,0}4}%
}}}}
\put(1801,-3931){\makebox(0,0)[lb]{\smash{{\SetFigFont{20}{24.0}{\rmdefault}{\mddefault}{\updefault}{\color[rgb]{0,0,0}5}%
}}}}
\put(1711,-4876){\makebox(0,0)[lb]{\smash{{\SetFigFont{20}{24.0}{\rmdefault}{\mddefault}{\updefault}{\color[rgb]{0,0,0}3}%
}}}}
\put(2251,-5056){\makebox(0,0)[lb]{\smash{{\SetFigFont{20}{24.0}{\rmdefault}{\mddefault}{\updefault}{\color[rgb]{0,0,0}4}%
}}}}
\put(1486,-5281){\makebox(0,0)[lb]{\smash{{\SetFigFont{20}{24.0}{\rmdefault}{\mddefault}{\updefault}{\color[rgb]{0,0,0}6}%
}}}}
\put(2116,-5641){\makebox(0,0)[lb]{\smash{{\SetFigFont{20}{24.0}{\rmdefault}{\mddefault}{\updefault}{\color[rgb]{0,0,0}7}%
}}}}
\put(496,-5236){\makebox(0,0)[lb]{\smash{{\SetFigFont{20}{24.0}{\rmdefault}{\mddefault}{\updefault}{\color[rgb]{0,0,0}$n=3$}%
}}}}
\put(496,-3796){\makebox(0,0)[lb]{\smash{{\SetFigFont{20}{24.0}{\rmdefault}{\mddefault}{\updefault}{\color[rgb]{0,0,0}$n=2$}%
}}}}
\put(496,-2356){\makebox(0,0)[lb]{\smash{{\SetFigFont{20}{24.0}{\rmdefault}{\mddefault}{\updefault}{\color[rgb]{0,0,0}$n=1$}%
}}}}
\put(6256,-4201){\makebox(0,0)[lb]{\smash{{\SetFigFont{20}{24.0}{\rmdefault}{\mddefault}{\updefault}{\color[rgb]{0,0,0}$5 \notin Y_{3}$}%
}}}}
\put(5806,-2851){\makebox(0,0)[lb]{\smash{{\SetFigFont{20}{24.0}{\rmdefault}{\mddefault}{\updefault}{\color[rgb]{0,0,0}$4 \in Y_{2}$}%
}}}}
\end{picture}%

%% file: kingman.pstex_t
\begin{picture}(0,0)%
\includegraphics{kingman.pstex}%
\end{picture}%
\setlength{\unitlength}{4144sp}%
\begingroup\makeatletter\ifx\SetFigFont\undefined%
\gdef\SetFigFont#1#2#3#4#5{%
  \reset@font\fontsize{#1}{#2pt}%
  \fontfamily{#3}\fontseries{#4}\fontshape{#5}%
  \selectfont}%
\fi\endgroup%
\begin{picture}(4524,2007)(2689,-2728)
\put(5041,-961){\makebox(0,0)[lb]{\smash{{\SetFigFont{20}{24.0}{\rmdefault}{\mddefault}{\updefault}{\color[rgb]{0,0,0}1}%
}}}}
\put(4726,-961){\makebox(0,0)[lb]{\smash{{\SetFigFont{20}{24.0}{\rmdefault}{\mddefault}{\updefault}{\color[rgb]{0,0,0}7}%
}}}}
\put(7021,-961){\makebox(0,0)[lb]{\smash{{\SetFigFont{20}{24.0}{\rmdefault}{\mddefault}{\updefault}{\color[rgb]{0,0,0}4}%
}}}}
\put(3286,-961){\makebox(0,0)[lb]{\smash{{\SetFigFont{20}{24.0}{\rmdefault}{\mddefault}{\updefault}{\color[rgb]{0,0,0}5}%
}}}}
\put(2881,-961){\makebox(0,0)[lb]{\smash{{\SetFigFont{20}{24.0}{\rmdefault}{\mddefault}{\updefault}{\color[rgb]{0,0,0}3}%
}}}}
\put(6166,-961){\makebox(0,0)[lb]{\smash{{\SetFigFont{20}{24.0}{\rmdefault}{\mddefault}{\updefault}{\color[rgb]{0,0,0}6}%
}}}}
\put(5581,-961){\makebox(0,0)[lb]{\smash{{\SetFigFont{20}{24.0}{\rmdefault}{\mddefault}{\updefault}{\color[rgb]{0,0,0}2}%
}}}}
\end{picture}%

%% file: kingman_levels.pstex_t
\begin{picture}(0,0)%
\includegraphics{kingman_levels.pstex}%
\end{picture}%
\setlength{\unitlength}{4144sp}%
\begingroup\makeatletter\ifx\SetFigFont\undefined%
\gdef\SetFigFont#1#2#3#4#5{%
  \reset@font\fontsize{#1}{#2pt}%
  \fontfamily{#3}\fontseries{#4}\fontshape{#5}%
  \selectfont}%
\fi\endgroup%
\begin{picture}(4932,3402)(2689,-4123)
\put(5041,-961){\makebox(0,0)[lb]{\smash{{\SetFigFont{20}{24.0}{\rmdefault}{\mddefault}{\updefault}{\color[rgb]{0,0,0}1}%
}}}}
\put(4726,-961){\makebox(0,0)[lb]{\smash{{\SetFigFont{20}{24.0}{\rmdefault}{\mddefault}{\updefault}{\color[rgb]{0,0,0}7}%
}}}}
\put(7021,-961){\makebox(0,0)[lb]{\smash{{\SetFigFont{20}{24.0}{\rmdefault}{\mddefault}{\updefault}{\color[rgb]{0,0,0}4}%
}}}}
\put(3286,-961){\makebox(0,0)[lb]{\smash{{\SetFigFont{20}{24.0}{\rmdefault}{\mddefault}{\updefault}{\color[rgb]{0,0,0}5}%
}}}}
\put(2881,-961){\makebox(0,0)[lb]{\smash{{\SetFigFont{20}{24.0}{\rmdefault}{\mddefault}{\updefault}{\color[rgb]{0,0,0}3}%
}}}}
\put(6166,-961){\makebox(0,0)[lb]{\smash{{\SetFigFont{20}{24.0}{\rmdefault}{\mddefault}{\updefault}{\color[rgb]{0,0,0}6}%
}}}}
\put(5581,-961){\makebox(0,0)[lb]{\smash{{\SetFigFont{20}{24.0}{\rmdefault}{\mddefault}{\updefault}{\color[rgb]{0,0,0}2}%
}}}}
\put(7606,-3616){\makebox(0,0)[lb]{\smash{{\SetFigFont{20}{24.0}{\rmdefault}{\mddefault}{\updefault}{\color[rgb]{0,0,0}$\phi_{3}$}%
}}}}
\put(7606,-3166){\makebox(0,0)[lb]{\smash{{\SetFigFont{20}{24.0}{\rmdefault}{\mddefault}{\updefault}{\color[rgb]{0,0,0}$\phi_{4}$}%
}}}}
\put(7606,-2716){\makebox(0,0)[lb]{\smash{{\SetFigFont{20}{24.0}{\rmdefault}{\mddefault}{\updefault}{\color[rgb]{0,0,0}$\phi_{1}$}%
}}}}
\put(7606,-2266){\makebox(0,0)[lb]{\smash{{\SetFigFont{20}{24.0}{\rmdefault}{\mddefault}{\updefault}{\color[rgb]{0,0,0}$\phi_{2}$}%
}}}}
\end{picture}%

%% file: feat_paint.pstex_t
\begin{picture}(0,0)%
\includegraphics{feat_paint.pstex}%
\end{picture}%
\setlength{\unitlength}{4144sp}%
\begingroup\makeatletter\ifx\SetFigFont\undefined%
\gdef\SetFigFont#1#2#3#4#5{%
  \reset@font\fontsize{#1}{#2pt}%
  \fontfamily{#3}\fontseries{#4}\fontshape{#5}%
  \selectfont}%
\fi\endgroup%
\begin{picture}(4707,3454)(2689,-4175)
\put(7381,-4066){\makebox(0,0)[lb]{\smash{{\SetFigFont{20}{24.0}{\rmdefault}{\mddefault}{\updefault}{\color[rgb]{0,0,0}$\phi_{5}$}%
}}}}
\put(5041,-961){\makebox(0,0)[lb]{\smash{{\SetFigFont{20}{24.0}{\rmdefault}{\mddefault}{\updefault}{\color[rgb]{0,0,0}1}%
}}}}
\put(4726,-961){\makebox(0,0)[lb]{\smash{{\SetFigFont{20}{24.0}{\rmdefault}{\mddefault}{\updefault}{\color[rgb]{0,0,0}7}%
}}}}
\put(7021,-961){\makebox(0,0)[lb]{\smash{{\SetFigFont{20}{24.0}{\rmdefault}{\mddefault}{\updefault}{\color[rgb]{0,0,0}4}%
}}}}
\put(3286,-961){\makebox(0,0)[lb]{\smash{{\SetFigFont{20}{24.0}{\rmdefault}{\mddefault}{\updefault}{\color[rgb]{0,0,0}5}%
}}}}
\put(2881,-961){\makebox(0,0)[lb]{\smash{{\SetFigFont{20}{24.0}{\rmdefault}{\mddefault}{\updefault}{\color[rgb]{0,0,0}3}%
}}}}
\put(6166,-961){\makebox(0,0)[lb]{\smash{{\SetFigFont{20}{24.0}{\rmdefault}{\mddefault}{\updefault}{\color[rgb]{0,0,0}6}%
}}}}
\put(5581,-961){\makebox(0,0)[lb]{\smash{{\SetFigFont{20}{24.0}{\rmdefault}{\mddefault}{\updefault}{\color[rgb]{0,0,0}2}%
}}}}
\put(7381,-2266){\makebox(0,0)[lb]{\smash{{\SetFigFont{20}{24.0}{\rmdefault}{\mddefault}{\updefault}{\color[rgb]{0,0,0}$\phi_{2}$}%
}}}}
\put(7381,-2716){\makebox(0,0)[lb]{\smash{{\SetFigFont{20}{24.0}{\rmdefault}{\mddefault}{\updefault}{\color[rgb]{0,0,0}$\phi_{3}$}%
}}}}
\put(7381,-3166){\makebox(0,0)[lb]{\smash{{\SetFigFont{20}{24.0}{\rmdefault}{\mddefault}{\updefault}{\color[rgb]{0,0,0}$\phi_{1}$}%
}}}}
\put(7381,-3616){\makebox(0,0)[lb]{\smash{{\SetFigFont{20}{24.0}{\rmdefault}{\mddefault}{\updefault}{\color[rgb]{0,0,0}$\phi_{4}$}%
}}}}
\end{picture}%

%% file: two_feat_paint2.pstex_t
\begin{picture}(0,0)%
\includegraphics{two_feat_paint2.pstex}%
\end{picture}%
\setlength{\unitlength}{4144sp}%
\begingroup\makeatletter\ifx\SetFigFont\undefined%
\gdef\SetFigFont#1#2#3#4#5{%
  \reset@font\fontsize{#1}{#2pt}%
  \fontfamily{#3}\fontseries{#4}\fontshape{#5}%
  \selectfont}%
\fi\endgroup%
\begin{picture}(4524,2101)(2689,-3500)
\put(6391,-3391){\makebox(0,0)[lb]{\smash{{\SetFigFont{20}{24.0}{\rmdefault}{\mddefault}{\updefault}{\color[rgb]{0,0,0}$p_{00}$}%
}}}}
\put(3061,-3391){\makebox(0,0)[lb]{\smash{{\SetFigFont{20}{24.0}{\rmdefault}{\mddefault}{\updefault}{\color[rgb]{0,0,0}$p_{10}$}%
}}}}
\put(4321,-3391){\makebox(0,0)[lb]{\smash{{\SetFigFont{20}{24.0}{\rmdefault}{\mddefault}{\updefault}{\color[rgb]{0,0,0}$p_{11}$}%
}}}}
\put(5446,-3391){\makebox(0,0)[lb]{\smash{{\SetFigFont{20}{24.0}{\rmdefault}{\mddefault}{\updefault}{\color[rgb]{0,0,0}$p_{01}$}%
}}}}
\end{picture}%

%% file: freq_paint.pstex_t
\begin{picture}(0,0)%
\includegraphics{freq_paint.pstex}%
\end{picture}%
\setlength{\unitlength}{4144sp}%
\begingroup\makeatletter\ifx\SetFigFont\undefined%
\gdef\SetFigFont#1#2#3#4#5{%
  \reset@font\fontsize{#1}{#2pt}%
  \fontfamily{#3}\fontseries{#4}\fontshape{#5}%
  \selectfont}%
\fi\endgroup%
\begin{picture}(4662,2398)(2689,-3797)
\put(7336,-2266){\makebox(0,0)[lb]{\smash{{\SetFigFont{20}{24.0}{\rmdefault}{\mddefault}{\updefault}{\color[rgb]{0,0,0}$\phi_{1}$}%
}}}}
\put(7336,-2716){\makebox(0,0)[lb]{\smash{{\SetFigFont{20}{24.0}{\rmdefault}{\mddefault}{\updefault}{\color[rgb]{0,0,0}$\phi_{2}$}%
}}}}
\put(7336,-3166){\makebox(0,0)[lb]{\smash{{\SetFigFont{20}{24.0}{\rmdefault}{\mddefault}{\updefault}{\color[rgb]{0,0,0}$\phi_{3}$}%
}}}}
\put(4996,-3616){\makebox(0,0)[lb]{\smash{{\SetFigFont{34}{40.8}{\rmdefault}{\mddefault}{\updefault}{\color[rgb]{0,0,0}$\vdots$}%
}}}}
\end{picture}%

%% file: draft.bbl
\begin{thebibliography}{17}
\newcommand{\enquote}[1]{``#1''}
\expandafter\ifx\csname natexlab\endcsname\relax\def\natexlab#1{#1}\fi
\expandafter\ifx\csname url\endcsname\relax
  \def\url#1{{\tt #1}}\fi
\expandafter\ifx\csname urlprefix\endcsname\relax\def\urlprefix{URL }\fi

\bibitem[{Aldous(1985)}]{aldous:1985:exchangeability}
Aldous, D. (1985).
\newblock \enquote{Exchangeability and related topics.}
\newblock {\em Ecole d'Et{\'e} de Probabilit{\'e}s de Saint-Flour XIIIÑ1983\/},
  1--198.

\bibitem[{Broderick et~al.(2012{\natexlab{a}})Broderick, Jordan, and
  Pitman}]{broderick:2012:beta}
Broderick, T., Jordan, M.~I., and Pitman, J. (2012{\natexlab{a}}).
\newblock \enquote{Beta processes, stick-breaking, and power laws.}
\newblock {\em Bayesian Analysis\/}, 7(2): 439--476.

\bibitem[{Broderick et~al.(2012{\natexlab{b}})Broderick, Jordan, and
  Pitman}]{broderick:2012:clusters}
--- (2012{\natexlab{b}}).
\newblock \enquote{Clusters and features from combinatorial stochastic
  processes.}
\newblock {\em Arxiv preprint arXiv:1206.5862\/}.

\bibitem[{Broderick et~al.(2011)Broderick, Mackey, Paisley, and
  Jordan}]{broderick:2011:combinatorial}
Broderick, T., Mackey, L., Paisley, J., and Jordan, M.~I. (2011).
\newblock \enquote{Combinatorial clustering and the beta negative binomial
  process.}
\newblock {\em Arxiv preprint arXiv:1111.1802\/}.

\bibitem[{Chen and Liu(1997)}]{chen:1997:statistical}
Chen, S.~X. and Liu, J.~S. (1997).
\newblock \enquote{Statistical applications of the {P}oisson-binomial and
  conditional {B}ernoulli distributions.}
\newblock {\em Statistica Sinica\/}, 7: 875--892.

\bibitem[{{De Finetti}(1931)}]{de_finetti:1931:funzione}
{De Finetti}, B. (1931).
\newblock \enquote{Funzione caratteristica di un fenomeno aleatorio.}
\newblock {\em Atti della R. Academia Nazionale dei Lincei, Serie 6.\/}, 4:
  251--299.
\newblock In {I}talian.

\bibitem[{Doshi et~al.(2009)Doshi, Miller, Van~Gael, and
  Teh}]{doshi:2009:variational}
Doshi, F., Miller, K.~T., Van~Gael, J., and Teh, Y.~W. (2009).
\newblock \enquote{Variational inference for the {I}ndian buffet process.}
\newblock In {\em Proceedings of the International Conference on Artificial
  Intelligence and Statistics\/}. Clearwater Beach, Florida, USA.

\bibitem[{Escobar(1994)}]{escobar:1994:estimating}
Escobar, M.~D. (1994).
\newblock \enquote{Estimating normal means with a Dirichlet process prior.}
\newblock {\em Journal of the American Statistical Association\/}, 268--277.

\bibitem[{Griffiths and Ghahramani(2006)}]{griffiths:2006:infinite}
Griffiths, T. and Ghahramani, Z. (2006).
\newblock \enquote{Infinite latent feature models and the {I}ndian buffet
  process.}
\newblock In {\em Advances in Neural Information Processing Systems\/}.
  Vancouver, B.C., Canada.

\bibitem[{Kallenberg(2002)}]{kallenberg:2002:foundations}
Kallenberg, O. (2002).
\newblock {\em Foundations of Modern Probability\/}.
\newblock Springer.

\bibitem[{Kingman(1978)}]{kingman:1978:representation}
Kingman, J.~F.~C. (1978).
\newblock \enquote{The representation of partition structures.}
\newblock {\em Journal of the London Mathematical Society\/}, 2(2): 374.

\bibitem[{Pitman(1995)}]{pitman:1995:exchangeable}
Pitman, J. (1995).
\newblock \enquote{Exchangeable and partially exchangeable random partitions.}
\newblock {\em Probability Theory and Related Fields\/}, 102(2): 145--158.

\bibitem[{Pitman(2006)}]{pitman:2006:combinatorial}
--- (2006).
\newblock {\em Combinatorial Stochastic Processes\/}, volume 1875 of {\em
  Lecture Notes in Mathematics\/}.
\newblock Berlin: Springer-Verlag.
\newline\urlprefix\url{http://bibserver.berkeley.edu/csp/april05/bookcsp.pdf}

\bibitem[{Teh and G{\"o}r{\"u}r(2009)}]{teh:2009:indian}
Teh, Y. and G{\"o}r{\"u}r, D. (2009).
\newblock \enquote{Indian buffet processes with power-law behavior.}
\newblock In {\em Advances in Neural Information Processing Systems\/}.
  Vancouver, B.C., Canada.

\bibitem[{Thibaux and Jordan(2007)}]{thibaux:2007:hierarchical}
Thibaux, R. and Jordan, M. (2007).
\newblock \enquote{Hierarchical beta processes and the {I}ndian buffet
  process.}
\newblock In {\em Proceedings of the International Conference on Artificial
  Intelligence and Statistics\/}. San Juan, Puerto Rico.

\bibitem[{Wang(1993)}]{wang:1993:number}
Wang, Y. (1993).
\newblock \enquote{On the number of successes in independent trials.}
\newblock {\em Statistica Sinica\/}, 3(2): 295--312.

\bibitem[{Zhou et~al.(2012)Zhou, Hannah, Dunson, and Carin}]{zhou:2012:beta}
Zhou, M., Hannah, L., Dunson, D., and Carin, L. (2012).
\newblock \enquote{Beta-negative binomial process and {P}oisson factor
  analysis.}
\newblock In {\em Proceedings of the International Conference on Artificial
  Intelligence and Statistics\/}. La Palma, Canary Islands.

\end{thebibliography}
